\documentclass[11pt, a4paper]{amsart}
\usepackage[latin1]{inputenc}
\usepackage[english]{babel}
\usepackage{amsmath, amsfonts, amssymb, amsthm}
\usepackage{enumerate}
\usepackage[mathscr]{euscript}
\usepackage{setspace}

\usepackage[arrow,matrix,curve]{xy}
\usepackage[pagebackref,colorlinks,linkcolor=blue,citecolor=blue,urlcolor=blue,hypertexnames=true,bookmarks=false]{hyperref}

\newtheorem{theorem}{Theorem}[section]
\newtheorem{lemma}[theorem]{Lemma}
\newtheorem{corollary}[theorem]{Corollary}

\theoremstyle{definition}
\newtheorem{definition}[theorem]{Definition}

\newtheorem{remark}[theorem]{Remark}

\newcommand{\N}{\mathbb{N}}
\newcommand{\Z}{\mathbb{Z}}
\newcommand{\Q}{\mathbb{Q}}
\newcommand{\R}{\mathbb{R}}
\newcommand{\C}{\mathbb{C}}
\newcommand{\K}{\mathbb{K}}		

\newcommand{\Shat}{\widehat{S}}

\newcommand{\Cuntz}[1]{\mathcal{O}_{#1}}
\newcommand{\JiangSu}{\mathcal{Z}}
\newcommand{\Aut}[1]{{\rm Aut}(#1)}

\newcommand{\id}[1]{{\rm id}_{#1}}

\newcommand{\cpt}[2]{\mathcal{K}_{#1}(#2)}
\newcommand{\bdd}[2]{\mathcal{L}_{#1}(#2)}

\newcommand{\lscal}[3]{\,_{#1}\langle #2, #3 \rangle}
\newcommand{\rscal}[3]{\langle #2, #3 \rangle_{#1}}

\newcommand{\KUmod}[2]{KU^{#1,{\rm mod}}_{#2}}
\newcommand{\KUort}[2]{KU^{#1}(#2)}
\newcommand{\KUbun}[2]{\mathscr{K}U^{#1}_{#2}}
\newcommand{\Kbun}[1]{\mathscr{K}^{#1}}
\newcommand{\KUobun}[2]{\mathscr{K}U^{#1}(#2)}
\newcommand{\Sec}[4]{{\rm Sec}_{#4}(#1,#2;#3)}
\newcommand{\Pb}{\mathcal{P}}

\newcommand{\TopPair}[1]{\mathcal{T}op^{\rm 2,cpt}_{#1}}

\newcommand{\CW}[1]{\mathcal{CW}^{\rm fin}_{#1}}

\newcommand{\RKK}[4]{\mathcal{R}KK_{#4}(#1; #2, #3)}

\newcommand{\Sp}{\mathcal{S}}
\newcommand{\St}[1]{{\rm St}(#1)}
\newcommand{\RMod}[1]{#1{\rm -Mod}}
\newcommand{\RLine}[1]{#1{\rm -Line}}

\DeclareMathOperator{\colim}{\rm colim}

\begin{document}
	
\title[A noncommutative model for higher twisted $K$-Theory]{A noncommutative model for higher twisted $K$-Theory}
\author{Ulrich Pennig}	
\thanks{U.P. was supported by the SFB 878}

\begin{abstract}
We develop a operator algebraic model for twisted $K$-theory, which includes the most general twistings as a generalized cohomology theory (i.e.\ all those classified by the unit spectrum $bgl_1(KU)$). Our model is based on strongly self-absorbing $C^*$-algebras. We compare it with the known homotopy theoretic descriptions in the literature, which either use parametrized stable homotopy theory or $\infty$-categories. We derive a similar comparison of analytic twisted $K$-homology with its topological counterpart based on generalized Thom spectra. Our model also works for twisted versions of localizations of the $K$-theory spectrum, like $KU[1/n]$ or $KU_{\Q}$.
\end{abstract}

\maketitle

\section{Introduction}
It was observed by Donovan and Karoubi in \cite{paper:DonovanKaroubi} that there is a Thom isomorphism (and Poincar\'{e} duality) in the absence of $K$-orientability if one considers $K$-theory with local coefficients, also known as twisted $K$-theory. Whereas ordinary operator algebraic $K$-theory involves matrix-valued functions on a space, twisted $K$-theory replaces those by sections of bundles of matrix algebras. Morita equivalence classes of the latter are classified by torsion elements in $H^3(X,\Z)$. Rosenberg \cite{paper:RosenbergCT} and Atiyah and Segal \cite{paper:AtiyahSegal} extended the definition in \cite{paper:DonovanKaroubi} by considering bundles of compact operators corresponding to non-torsion twists. The observation that twisted $K$-groups classify $D$-brane charges sparked interactions with string theory \cite{paper:BCMMS}. The equivariant twisted $K$-theory of a Lie group $G$ with respect to the adjoint action on itself was proven to be closely related to the Verlinde ring of positive energy representations of the loop group $LG$ by Freed, Hopkins and Teleman \cite{paper:FreedHopkinsTelemanI}. 

From a homotopy theoretic viewpoint the twists of complex topological $K$-theory $KU$ are classified by the generalized cohomology theory $gl_1(KU)$ associated to its unit spectrum. Its first group $[X,BGL_1(KU)]$ contains $[X, BBU(1)] \cong H^3(X,\Z)$ as a direct summand, which appears in the applications mentioned in the last paragraph. The existence of more general twistings was already pointed out in \cite{paper:AtiyahSegal}, but they were neglected since they had no obvious geometric interpretation in terms of bundles of operator algebras at that time. Nevertheless, in the setting of parametrized stable homotopy theory as developed by May and Sigurdsson \cite{book:MaySigurdsson} it is just as easy to deal with the whole space $BGL_1(KU)$ as it is to deal with $BBU(1)$. An equivalent approach to twisted generalized (co)homology theories was worked out by Ando, Blumberg, Gepner, Hopkins and Rezk in  \cite{paper:Ando, paper:ABGH}. It uses a generalized Thom spectrum associated to a map $X \to KU$-Line, where $KU$-Line is an $\infty$-groupoid with the homotopy type $BGL_1(KU)$.  

As elegant as these approaches are, they give no hint towards an operator algebraic interpretation as in the case of ordinary twists. It is the goal of this paper to show that in fact the ``higher twists'' appear very naturally in the setting of operator algebras if one considers bundles of stabilized strongly self-absorbing $C^*$-algebras instead of bundles of compact operators. The class of these algebras was studied first by Toms and Winter \cite{paper:TomsWinter, paper:WinterZStable} and Dadarlat and Winter \cite{paper:DadarlatKK1}. 

Their definition was motivated by the following observations: Certain algebras play a cornerstone role in Elliott's classification program of separable, nuclear, simple $C^*$-algebras: Tensorial absorption of the Cuntz algebra $\Cuntz{\infty}$ detects purely infiniteness \cite[Thm.\ 7.2.6]{book:RoerdamStoermer}. Toms and Winter conjecture that in case of stably finite $C^*$-algebras, tensorial absorption of the Jiang-Su algebra $\JiangSu$ detects finite nuclear dimension \cite{paper:TomsAH}. Similarly, the Cuntz algebra $\Cuntz{2}$ tensorially absorbs another algebra if and only if the latter is simple, separable, unital and nuclear \cite[Thm.\ 7.1.2]{book:RoerdamStoermer}. The definition of strongly self-absorbing isolates a property that all of these algebras share: They tensorially absorb themselves in a very strong sense (see Def.\ \ref{def:ssa}).

Surprisingly, this property has a lot of interesting topological and homotopy theoretic consequences: Given a strongly self-absorbing $C^*$-algebra $D$, the functor $X \mapsto K_*(C(X) \otimes D)$ is a multiplicative generalized cohomology theory, which can be represented by a commutative symmetric ring spectrum $KU^D_{\bullet}$ \cite{paper:DadarlatP2}. In particular, $KU^{\Cuntz{\infty}}_{\bullet}$, $KU^{\JiangSu}_{\bullet}$ both yield complex topological $K$-theory and if $D$ is an infinite UHF-algebra, $KU^D_{\bullet}$ represents localizations of $K$-theory, such as $KU[1/p]$. It was proven in \cite{paper:DadarlatP1, paper:DadarlatP2} that $B\Aut{D \otimes \K}$ is an infinite loop space, which comes with a map $B\Aut{D \otimes \K} \to BGL_1(KU^D)$. In case $D = \C$ we recover $B\Aut{\K} = BPU(H) \simeq BBU(1)$. In all other cases the map $\Aut{D \otimes \K} \to GL_1(KU^D)$ is an isomorphism on $\pi_k$ for $k > 0$ and the inclusion $K_0(D)^{\times}_+ \to K_0(D)^{\times}$ on $\pi_0$. In particular, it is an equivalence if $D$ is strongly self-absorbing and purely infinite \cite[Thm.\ 4.6]{paper:DadarlatP2}. 

Based on these observations we will show in the current paper that the operator algebraic $K$-theory of section algebras of bundles with stabilized strongly self-absorbing fibers is indeed a valid model for higher twisted $K$-theory: In Section $2$ we first construct a $KU^D_{\bullet}$-module spectrum $\KUmod{D}{\bullet}$, on which the automorphism group $\Aut{D \otimes \K}$ acts via maps of symmetric spectra. We then form a universal bundle of symmetric $KU^D_{\bullet}$-module spectra $\KUbun{D}{\bullet}$ over $B\Aut{D \otimes \K}$. Sections of this bundle pulled back to a compact space $X$ will provide the topological version of twisted $K$-theory. We prove that it has all the properties of a twisted generalized cohomology theory on pairs of compact topological spaces in Thm.\ \ref{thm:Twisted_K} and compare it with its analytic counterpart, i.e.\ the operator algebraic $K$-theory of the corresponding section algebra. 

In the next section we compare our definition with the twisted $K$-theory obtained from parametrized stable homotopy theory. We will not develop the analogue of the $qf$-model structure from \cite{book:MaySigurdsson} for symmetric spectra. Instead we consider an orthogonal ring spectrum $V \mapsto KU(V)$, which is equivalent to $KU^{\C}_{\bullet}$ when restricted to the symmetric group action. For any $C^*$-algebra $A$ there is a $KU$-module spectrum $KU^A$, which is equivalent as a symmetric spectrum to $\KUmod{D}{\bullet}$ for $A = D \otimes \K$. This allows us to compare our previous definition with the parametrized cohomology theories from \cite[Ch.\ 20]{book:MaySigurdsson}, which now merely amounts to citing the right theorems. In case $D = \{\C, \Cuntz{\infty}, \JiangSu\}$ no information is lost when switching to orthogonal spectra. However, there seems to be no obvious orthogonal counterpart for the symmetric ring spectrum $KU^D_{\bullet}$ in the case of general $D$, so we only retain the $KU^{\C}_{\bullet}$-module structure during this process.

Topological twisted $K$-homology is dual to twisted $K$-theory and can be defined as the homotopy groups of a generalized Thom spectrum $Mf_{\bullet}$ associated to a map $f \colon X \to BGL_1(KU)$. Depending on the setup, $Mf_{\bullet}$ can be constructed as an $\infty$-categorical colimit \cite{paper:Ando, paper:ABGH}, via the bar construction \cite[Ch.\ 23]{book:MaySigurdsson} or as a smash product. We obtain $Mf_{\bullet}$ from the bundle of spectra $f^*\KUbun{D}{\bullet}$ by collapsing its zero section (Def.~\ref{def:twisted_K_homology}) in the spirit of \cite{paper:Wang1}. Analytic $K$-homology on the other hand is defined via $KK$-theory. To compare the two functors on finite CW-complexes, we generalize the proof in \cite{paper:BaumHigsonSchick} to the twisted case. In particular, we obtain an intermediate homology theory from framed bordism mapping to both of the theories. The map to analytic $K$-homology relies on a twisted version of Poincar\'{e} duality for bundles of $C^*$-algebras \cite{paper:EEK} and a twisted bordism invariant index map. This also yields a new proof of a similar result for ordinary twists, which was proven by Wang \cite{paper:Wang1} for smooth manifolds and for finite CW-complexes \cite{paper:Wang2}. 

In the last part of this paper we give an explicit construction of higher twists over spaces of the form $X = \Sigma Y$, where we have $[\Sigma Y, BGL_1(KU^D)] \cong [Y, \Aut{D \otimes \K}] \cong K^0(Y)^{\times}_+$. This yields continuous $C(X)$-algebras with fibers $D \otimes \K$ representing all possible higher twists in the suspension case, in particular for spheres. We then calculate the higher twisted $K$-groups of these spaces in terms of the $K$-theory of $Y$ using a Mayer-Vietoris argument. 
\vspace{0.4cm}
\paragraph{\it Acknowledgements}
The author would like to thank Marius Dadarlat, Thomas Schick and Michael Joachim for many enlightening discussions and Qiaochu Yuan for answering a question on mathoverflow.net concerning the Mayer-Vietoris sequence for twisted $R$-homology. Part of this paper was finished during the trimester ``Noncommutative geometry and its applications'' at the Hausdorff institute for Mathematics in Bonn. The author would like to thank the organizers of this workshop and the staff at the HIM for the hospitality.

\section{Twisted $K$-theory} 

\subsection{Strongly self-absorbing $C^*$-algebras and $KU^D_{\bullet}$}
The notion of strongly self-absorbing $C^*$-algebras was introduced by Toms and Winter in \cite{paper:TomsWinter}. We will use the following definition, which is closer to topological applications and equivalent to the original one by \cite[Thm.\ 2.2]{paper:DadarlatKK1} and \cite{paper:WinterZStable}.
\begin{definition} \label{def:ssa}
A unital $C^*$-algebra $D$ is called \emph{strongly self-absorbing} if it is separable and there exists a $*$-isomorphism $\psi \colon D \to D \otimes D$ and a path of unitaries $u \colon [0,1) \to U(D \otimes D)$ such that for all $d \in D$
\[
	\lim_{t \to 1} \lVert \psi(d) - u_t(d \otimes 1)u_t^* \rVert = 0\ .
\]
\end{definition}

As mentioned in the introduction the main examples of strongly self-absorbing $C^*$-algebras are the Cuntz algebras $\Cuntz{\infty}$ and $\Cuntz{2}$, the Jiang-Su algebra $\JiangSu$, infinite UHF algebras and tensor products of those. Their most important properties are summarized in \cite[Thm.\ 2.1]{paper:DadarlatP1}. 

In the following $D$ will always denote a strongly self-absorbing $C^*$-algebra. Let $X$ be a compact Hausdorff space. The inverse isomorphism $\psi^{-1} \colon D \otimes D \to D$ together with the diagonal homomorphism $C(X) \to C(X \times X) \cong C(X) \otimes C(X)$ equips $K_{\ast}(C(X) \otimes D)$ with a ring structure. Since $\Aut{D}$ is contractible by \cite[Thm.\ 2.3]{paper:DadarlatP1}, the homotopy invariance of $K$-theory implies that this multiplication does not depend on the choice of $\psi$. Likewise we obtain from the above property that the projection $[1_{C(X) \otimes D}] \in K_0(C(X) \otimes D)$ is the unit element of this graded ring. 

The definition of the spectrum $KU^D_{\bullet}$ representing $X \mapsto K_*(C(X) \otimes D)$ will require graded $C^*$-algebras. We refer the reader to \cite[VI.14]{book:Blackadar} for an introduction. All tensor products that appear in the following are considered to be graded. Following \cite{paper:DadarlatP2} we denote the $C^*$-algebra $C_0(\R)$ graded by odd and even functions by $\Shat$. This is a coassociative, counital, coalgebra with comultiplication $\Delta \colon \Shat \to \Shat \otimes \Shat$ and counit $\epsilon \colon \Shat \to \C$ \cite{paper:FunctorialKSpectrum}, \cite{paper:Joachim}. It has the universal property that for any graded $\sigma$-unital $C^*$-algebra $B$ there is a correspondence between essential graded $*$-homomorphisms $\varphi \colon \Shat \to B$ and unbounded, self-adjoint, regular, odd operators $D \colon A \to A$ \cite{paper:Trout}. Moreover, let $\C\ell_1$ be the complex Clifford algebra spanned by the even element $1$ and the odd element $c$ with $c^2 = 1$. 

Each strongly self-absorbing $C^*$-algebra $D$ gives rise to a commutative symmetric ring spectrum $KU^D_{\bullet}$ \cite{paper:DadarlatP2}: Equip $D \otimes \K$ with the trivial grading and consider the sequence of topological spaces
\[
	KU^D_n = \hom_{\rm gr}(\widehat{S}, (\C\ell_1 \otimes D \otimes \K)^{\otimes n})\ ,
\]
where each of the homomorphism spaces is equipped with the pointwise norm topology and pointed by the zero homomorphism. Let $\mu_{m,n}$ be the following multiplication map
\[
	\mu_{m,n} \colon KU^D_m \wedge KU^D_n \to KU^D_{m+n} \quad ; \quad \varphi \wedge \psi \mapsto (\varphi \otimes \psi) \circ \Delta 
\]
By the universal property of $\Shat$ the unbounded, self-adjoint, regular, odd multiplier $t \mapsto t\,c$ on $C_0(\R,\C\ell_1)$ defines a $*$-homomorphism $\Shat \to C_0(\R,\C\ell_1)$. Tensor product with $\C \to D \otimes \K$ given by $1 \mapsto 1 \otimes e$ induces a map $\widehat{\eta_1} \colon \Shat \to C_0(\R, \C\ell_1 \otimes D \otimes \K)$, which in turn induces $\eta_1 \colon S^1 \to KU^D_1$. Forming products we obtain $\eta_n \colon S^n \to KU^D_n$. It was shown in \cite[Thm.\ 4.2]{paper:DadarlatP2} that $(KU^D_{\bullet}, \mu_{\bullet,\bullet}, \eta_{\bullet})$ gives a commutative symmetric ring spectrum, which represents the functor $X \mapsto K_{\bullet}(C(X) \otimes D)$.

We will also need the definition of a module spectrum over $KU^D_{\bullet}$.
\begin{definition} \label{def:mod_spectrum}
Let $(R_{\bullet}, \mu_{\bullet,\bullet}, \eta_{\bullet})$ be a commutative symmetric ring spectrum. A \emph{module spectrum} is a sequence of pointed spaces $(M_n)_{n \geq 0}$ with basepoint preserving continuous left action of the symmetric group $\Sigma_n$ on $M_n$ for each $n \geq 0$ together with $\Sigma_m \times \Sigma_n$-equivariant action maps $\alpha_{m,n} \colon M_m \wedge R_n \to M_{m+n}$ for all $n,m \geq 0$, such that the following diagrams commute:
\[
	\xymatrix@C=2cm{
	M_{\ell} \wedge R_m \wedge R_n \ar[r]^-{\alpha_{\ell,m} \wedge \id{R_n}}  \ar[d]_-{\id{M_{\ell}} \wedge \mu_{m,n}} & M_{\ell + m} \ar[d]^-{\alpha_{\ell+m,n}} \wedge R_n \\
	M_{\ell} \wedge R_{m+n} \ar[r]_-{\alpha_{\ell,m+n}} & M_{\ell + m + n}
	} \qquad 
	\xymatrix@C=2cm{
	M_n \wedge S^0 \ar[dr] \ar[r]^-{\id{M_n} \wedge \eta_0} & M_n \wedge R_0 \ar[d]^-{\alpha_{n,0}}\\
	& M_n
	}
\]
where the map $M_n \wedge S^0 \to M_n$ in the right diagram is the canonical one.
\end{definition}

We now define the module spectrum $\KUmod{D}{\bullet}$ over $KU^D_{\bullet}$, which carries an $\Aut{D \otimes \K}$-action that is compatible with the action maps.

\begin{definition} \label{def:KUD_spectrum}
Let $D$ be a strongly self-absorbing $C^*$-algebra. Let $\KUmod{D}{\bullet}$ be the following sequence of spaces
\[
	\KUmod{D}{n} = \hom_{\rm gr}(\widehat{S}, D \otimes \K \otimes (\C\ell_1 \otimes D \otimes \K)^{\otimes n})\ ,
\]
where as above the graded homomorphisms are equipped with the point-norm topology and $D \otimes \K$ is considered to be trivially graded. Let 
\[
	\alpha_{m,n} \colon \KUmod{D}{m} \wedge KU^D_n \to \KUmod{D}{m+n} \quad ; \quad \varphi \wedge \psi \mapsto (\varphi \otimes \psi) \circ \Delta
\] 
\end{definition}
Observe that the space $\KUmod{D}{n}$ carries a natural left action of the group $\Aut{D \otimes \K}$ in the following way: Let $\beta \in \Aut{D \otimes \K}$ and $\varphi \in \KUmod{D}{n}$, then we set $\beta \cdot \varphi = (\beta \otimes \id{(\C\ell_1 \otimes D \otimes \K)^{\otimes n}}) \circ \varphi$. Moreover there is a natural $\Sigma_n$-action on $\KUmod{D}{n}$ by permuting the $n$ factors of the graded tensor product $(\C\ell_1 \otimes D \otimes \K)^{\otimes n}$ using the Koszul sign rule just as in $KU^D_{\bullet}$.

\begin{theorem} \label{thm:KU_module}
The pair $(\KUmod{D}{\bullet}, \alpha_{\bullet,\bullet})$ forms a module spectrum over the ring spectrum $KU^D_{\bullet}$, which is equivalent to $KU^D_{\bullet}$ as a module spectrum. The module structure is compatible with the action of $\Aut{D \otimes \K}$ described above. All structure maps $\KUmod{D}{n} \to \Omega \KUmod{D}{n+1}$ are weak homotopy equivalences for $n \geq 1$ and the coefficients are given by 
\[
	\pi_i(\KUmod{D}{\bullet}) = K_i(D)\ .
\] 
\end{theorem}

\begin{proof} \label{pf:KU_module}
It is clear that the $\Sigma_n$-action preserves the basepoint of $\KUmod{D}{n}$. It is a consequence of the coassociativity of $\Delta \colon \Shat \to \Shat \otimes \Shat$ that the associativity diagram in Def.~\ref{def:KUD_spectrum} commutes. Since $\eta_0$ sends the non-basepoint of $S^0$ to the counit $\epsilon \colon \Shat \to \C$, the unitality diagram  also commutes. The group $\Sigma_n$ acts on $\KUmod{D}{n}$ from the left. This implies the equivariance of $\alpha_{m,n}$ with respect to the $\Sigma_m \times \Sigma_n$-action on both sides. This proves that $(\KUmod{D}{\bullet}, \alpha_{\bullet,\bullet})$ is in fact a symmetric $KU^D_{\bullet}$-module spectrum. 

Let $\beta \in \Aut{D \otimes \K}$, $\varphi \in \KUmod{D}{m}$, $\psi \in KU^D_{n}$, then we have 
\[
	\alpha_{m,n}((\beta \cdot \varphi) \wedge \psi) = (\beta \otimes \id{}) \circ (\varphi \otimes \psi) \circ \Delta = \beta \cdot \alpha_{m,n}(\varphi \wedge \psi)\ ,
\]
which is the claimed compatibility with the action. To see that $\KUmod{D}{\bullet}$ is equivalent to $KU^D_{\bullet}$ as a module spectrum, consider the $\Sigma_n$-equivariant map $\theta_n \colon KU^D_{n} \to \KUmod{D}{n}$ given by $\theta_n(\varphi)(f) = 1 \otimes e \otimes \varphi(f)$ for $f \in \Shat$, $\varphi \in KU^D_n$. Associativity of the graded tensor product implies $\alpha_{m,n}(\theta_m(\varphi) \wedge \psi) = \theta_{m+n}(\mu_{m,n}(\varphi \wedge \psi))$ for $\varphi \in KU^D_m$, $\psi \in KU^D_{n}$, i.e.\ $\theta_n$ is a map of $KU^D_{\bullet}$-module spectra. Since $D$ is strongly self-absorbing, the homomorphism $D \otimes \K \to (D \otimes \K)^{\otimes 2}$ given by $a \mapsto 1 \otimes e \otimes a$ is homotopic to an isomorphism. Therefore $\theta_n$ is a weak equivalence for $n \geq 1$. The statement about the structure maps and the coefficients follows from \cite[Thm.\ 4.2]{paper:DadarlatP2}. This implies that $\KUmod{D}{\bullet}$ is semistable, therefore $\theta_{\bullet}$ is actually an equivalence of symmetric spectra. 
\end{proof}

\subsection{Bundles of $KU^D_{\bullet}$-module spectra}
The advantage of $\KUmod{D}{\bullet}$ over $KU^D_{\bullet}$ is that it carries a natural left action of the group $\Aut{D \otimes \K}$, which enables us to form a bundle of symmetric spectra over $B\Aut{D \otimes \K}$.
\begin{definition} \label{def:KU_bundle}
Let $E\Aut{D \otimes \K} \to B\Aut{D \otimes \K}$ be the universal principal $\Aut{D \otimes \K}$-bundle and denote the left action of $\Aut{D \otimes \K}$ on $\KUmod{D}{n}$ by $\lambda_n$. We define
\[
	\KUbun{D}{n} = E\Aut{D \otimes \K} \times_{\lambda_n} \KUmod{D}{n}\ ,
\]
let $\pi_n \colon \KUbun{D}{n} \to B\Aut{D \otimes \K}$ be the projection map and let $\sigma_n \colon B\Aut{D \otimes \K} \to \KUbun{D}{n}$ be the zero section. Note that $\Aut{D \otimes \K}$ also acts on the loop space $\Omega \KUmod{D}{\bullet}$ preserving its basepoint. We denote the corresponding bundle by $\Omega \KUbun{D}{\bullet}$ and we have structure maps $\kappa_n \colon \KUbun{D}{n} \to \Omega\KUbun{D}{n+1}$ induced by $\KUmod{D}{n} \to \Omega\KUmod{D}{n+1}$.
\end{definition}

Let $\TopPair{D}$ be the category of triples $(X,B,f)$, where $X$ is a compact Hausdorff space, $B \subset X$ is a closed subspace and $f \colon X \to B\Aut{D \otimes \K}$ is a continuous map classifying a principal bundle $\Pb_f = f^*E\Aut{D \otimes \K}$ over $X$ with an associated bundle of $C^*$-algebras $\mathcal{A} = \mathcal{A}_f \to X$ with fiber $D \otimes \K$. A morphism in $\TopPair{D}$ is given by a pair $(\varphi, \hat{\varphi})$, where $\varphi$ is a continuous map of pairs $(X,B) \to (X',B')$ and $\hat{\varphi}$ is an isomorphism $\Pb_f \to \Pb_{f'}$. Define $c(X,B) = X \amalg B \times [0,1) / \!\!\sim$, where the equivalence relation identifies $(b,0) \in B \times [0,1)$ with $b \in X$. The bundle $\Pb_f$ extends in a canonical way to a bundle $c\,\Pb_f \to c(X,B)$ with an associated bundle of $C^*$-algebras $c\,\mathcal{A} \to c(X,B)$. We have a short exact sequence of section algebras
\begin{equation} \label{eqn:short_exact}
	0 \to SC(B,\left.\mathcal{A}\right|_B) \to C_0(c(X,B), c\,\mathcal{A}) \to C(X,\mathcal{A}) \to 0
\end{equation}
where $SC(B,\left.\mathcal{A}\right|_B) = C_0((0,1)) \otimes C(B,\left.\mathcal{A}\right|_B)$ is the suspension. We define the ordinary $K$-theory of the pair $(X,B)$ with coefficients in $D$ by $K^i_D(X,B) = K_i(C_0(c(X,B)) \otimes D)$. This is justified by the fact that for $D = \C$ this group is the $K$-theory of the mapping cone of the inclusion $B \subseteq X$ relative to its tip. Since $D$ is strongly self-absorbing $K^i_D(X,B)$ is a ring. Twisted $K$-theory is a functor from $\TopPair{D}$ to graded abelian groups in such a way that it maps a triple $(X,B,f)$ to a module over the ordinary $K$-theory of $(X,B)$ with coefficients in $D$. 

\begin{definition}\label{def:twisted_K}
Let $(X,B,f) \in {\rm Ob}(\TopPair{D})$, let $\Pb = f^*E\Aut{D \otimes \K} \to X$ and let $\Sec{X}{B}{\Pb}{n}$ be the space of all pairs $(\tau, H)$, where $\tau \colon X \to f^*\KUbun{D}{n}$ is a section and $H \colon B \times I \to (\left.f\right|_B \times \id{I})^*\KUbun{D}{n}$ is another section which satisfies $H_0 = \left.\tau\right|_B$ and $H_1(a) = \sigma_n(f(a))$ (where $\sigma_n$ denotes the zero section from Def.\ \ref{def:KU_bundle}). This space is pointed by $(\sigma_n \circ f, {\rm const}_{\sigma})$, where ${\rm const}_{\sigma}$ is the constant homotopy on $\left.\sigma_n \circ f\right|_B$. Observe that the structure maps of $\KUbun{D}{n}$ induce a continuous map
\[
	\Sec{X}{B}{\Pb}{n} \to \Omega \Sec{X}{B}{\Pb}{n+1} 
\]
Define the \emph{twisted $K$-group} $K^i_{\Pb}(X,B)$ of $(X,B,f)$ by
\[
	K^i_{\Pb}(X,B) = \colim_{n}\pi_{i+n}(\Sec{X}{B}{\Pb}{n})\ .
\]
Let $K^i_{\Pb}(X) = K^i_{\Pb}(X, \emptyset)$. Note that $\Sec{X}{\emptyset}{\Pb}{n}$ is just the space of sections of $f^*\KUbun{D}{n}$ pointed by the zero section.
\end{definition}

\begin{theorem} \label{thm:Twisted_K}
Let $(X,B,f) \in {\rm Ob}(\TopPair{D})$ and let $\Pb = f^*E\Aut{D \otimes \K}$. Twisted $K$-theory has the following properties:
\begin{enumerate}[(a)]
	\item \label{it:a} The structure maps $\kappa_n$ of the bundle $\KUbun{D}{n}$ are weak equivalences for $n \geq 1$. Therefore $K^m_{\Pb}(X,B) \cong \pi_{m+1}(\Sec{X}{B}{\Pb}{1})$.
	\item \label{it:b} $K^{\bullet} \colon \TopPair{D} \to {\rm GrAb}$ is a homotopy invariant contravariant functor to graded abelian groups, such that $K^{\bullet}_{\Pb}(X,B)$ is a module over $K^{\bullet}_D(X,B)$. 
	\item \label{it:c} Topological and analytic twisted $K$-theory are naturally isomorphic, i.e.\ we have $K^m_{\Pb}(X,B) \cong K_m(C_0(c(X,B), c\,\mathcal{A}))$. In particular, there is a natural isomorphism $K^m_{\Pb}(X) \cong K_m(C(X,\mathcal{A}))$ and a trivialization of $\Pb$ induces an isomorphism $K^m_{\Pb}(X,B) \cong K^m_D(X,B)$.
	\item \label{it:d} Let $i \colon (X, \emptyset) \to (X,B)$ and $j \colon B \to X$ be given by inclusion. There is a six-term exact sequence of the form
	\[
		\xymatrix{
			K^0_{\Pb}(X,B) \ar[r]^-{i^*} & K^0_{\Pb}(X) \ar[r]^-{j^*} & K^0_{\left.\Pb\right|_B}(B) \ar[d]^-{\partial_{01}} \\
			K^1_{\left.\Pb\right|_B}(B) \ar[u]^-{\partial_{10}} & K^1_{\Pb}(X) \ar[l]_-{j^*} & K^1_{\Pb}(X,B) \ar[l]_-{i^*}
		}
	\]
	\item \label{it:ex} $K^m_{\Pb}$ satifies excision in the sense that for every open set $U \subset B$ we have 
	\[
		K^m_{\Pb}(X,B) \cong K^m_{\Pb}(X \setminus U, B \setminus U)\ .
	\]
	\item \label{it:e} Let $X = V_0 \cup V_1$ for closed subsets $V_i$, such that their interiors still cover $X$. Let $i_k \colon V_k \to X$ and let $j_k \colon V_0 \cap V_1 \to V_k$ be given by inclusion. Then there is a six-term Mayer-Vietoris sequence:
	\[
		\xymatrix@C=1.7cm{
			K^0_{\Pb}(X) \ar[r]^-{(i_{0}^*, i_{1}^*)} & K^0_{\left.\Pb\right|_{V_0}}\!(V_0) \oplus K^0_{\left.\Pb\right|_{V_1}}\!(V_1) \ar[r]^-{j_0^* - j_1^*} & K^0_{\left.\Pb\right|_{V_0 \cap V_1}}\!(V_0 \cap V_1) \ar[d] \\
			K^1_{\left.\Pb\right|_{V_0 \cap V_1}}\!(V_0 \cap V_1) \ar[u] & K^1_{\left.\Pb\right|_{V_0}}\!(V_0) \oplus K^1_{\left.\Pb\right|_{V_1}}\!(V_1) \ar[l]_-{j_0^* - j_1^*} & K^1_{\Pb}(X) \ar[l]_-{(i_0^*, i_1^*)}
		}
	\] 
\end{enumerate}
\end{theorem}

\begin{proof} \label{pf:Twisted_K}
By the same argument as in the proof of \cite[Thm.\ 4.2]{paper:DadarlatP2} the maps $\KUmod{D}{n} \to \Omega \KUmod{D}{n+1}$ are weak equivalences for $n \geq 1$. The first statement of (\ref{it:a}) follows from the long exact sequence for the fibration $\KUmod{D}{n} \to \KUbun{D}{n} \to B\Aut{D \otimes \K}$. 

Let $(\tau,H) \in \Sec{X}{B}{\Pb}{n}$, let $C_n = (\C\ell_1 \otimes D \otimes \K)^{\otimes n}$. The section $\tau$ yields an element in $\hom_{\rm gr}(\Shat, C(X,\mathcal{A}) \otimes C_n)$. Similarly, $H$ yields a homomorphism in $\hom_{\rm gr}(\Shat, C_0(B \times [0,1),\pi_B^*\mathcal{A}) \otimes C_n)$. The algebra $C_0(c(X,B), c\,\mathcal{A})$ is the pullback of $C(X,\mathcal{A})$ and $C_0(B \times [0,1), \pi_B^*\mathcal{A})$ along the restriction maps to $C(B,\left.\mathcal{A}\right|_B)$. Thus, the definition of $\Sec{X}{B}{\Pb}{n}$ ensures that both homomorphisms piece together to form a point in 
\[
	K_n = \hom_{\rm gr}(\Shat, C_0(c(X,B), c\,\mathcal{A}) \otimes (\C\ell_1 \otimes D \otimes \K)^{\otimes n})\ .
\]
This construction provides a homeomorphism between $\Sec{X}{B}{\Pb}{n}$ and $K_n$ with respect to which $\Sec{X}{B}{\Pb}{n} \to \Omega\,\Sec{X}{B}{\Pb}{n+1}$ translates to a map $K_n \to \Omega K_{n+1}$. The latter is essentially given by Bott periodicity and tensor product with the projection $1 \otimes e \in D \otimes \K$ as in the proof of \cite[Thm.\ 4.2]{paper:DadarlatP2}. Thus, $K_n \to \Omega K_{n+1}$ is a weak equivalence for $n \geq 1$, which proves the isomorphism in (\ref{it:a}). By \cite[Thm.\ 4.7]{paper:Trout} and Bott periodicity, we have natural isomorphisms
\begin{align*}
	\pi_{m+1}(K_1) & = \pi_0(\hom_{\rm gr}(\Shat, S^mC_0(c(X,B), c\,\mathcal{A}) \otimes C_0(\R,\C\ell_1) \otimes D \otimes \K)) \\
	& \cong K_0(S^mC_0(c(X,B), c\,\mathcal{A})) \cong K_m(C_0(c(X,B), c\,\mathcal{A}))
\end{align*}
If $B = \emptyset$ we can identify $c(X,B) = X$ and $c\,\mathcal{A} = \mathcal{A}$. This proves (\ref{it:c}). Since $K_m$ is homotopy invariant and $K_{\bullet}(C_0(c(X,B),c\,\mathcal{A}))$ is a module over $K_{\bullet}(C_0(c(X,B)) \otimes D)$, it also proves (\ref{it:b}). The algebra $C_0(c(X,B), c\,\mathcal{A})$ fits into a short exact sequence
\begin{equation} \label{eqn:cone_seq}
	0 \to SC(B, \left.\mathcal{A}\right|_B) \to C_0(c(X,B), c\,\mathcal{A}) \to C(X,\mathcal{A}) \to 0\ .
\end{equation}
Using the isomorphism from (\ref{it:c}), the six-term exact sequence associated to the above short exact sequence is of the form given in (\ref{it:d}), but we have to check that the map $K_0(C(X,\mathcal{A})) \to K_1(SC(B,\left.\mathcal{A}\right|_B)) \cong K_0(C(B, \left.\mathcal{A}\right|_B))$ (with the indicated isomorphism given by Bott periodicity) coincides with the one induced by the inclusion $B \to X$. We can restrict (\ref{eqn:cone_seq}) to the sequence
\begin{equation} \label{eqn:snd_cone_seq}
	0 \to SC(B,\left.\mathcal{A}\right|_B) \to C_0(B \times [0,1), \left.c\,\mathcal{A}\right|_{B \times [0,1)}) \to C(B, \left.\mathcal{A}\right|_B) \to 0\ .
\end{equation}
in which the middle algebra is contractible. Naturality of the boundary map yields the commutative diagram
\[
	\xymatrix{
		K_0(C(X,\mathcal{A})) \ar[d] \ar[r] & K_1(SC(B,\left.\mathcal{A}\right|_B)) \ar@{=}[d] \\
		K_0(C(B, \left.\mathcal{A}\right|_B)) \ar[r]_-{\cong} & K_1(SC(B,\left.\mathcal{A}\right|_B))
	}
\]
where the vertical arrow is induced by the inclusion and the isomorphism from the six-term sequence associated to (\ref{eqn:snd_cone_seq}) coincides with the Bott isomorphism. This finishes the proof of (\ref{it:d}). The Mayer-Vietoris sequence (\ref{it:e}) is a consequence of the isomorphism in (\ref{it:c}) and the fact that $C(X,\mathcal{A})$ fits into the pullback diagram
\[
	\xymatrix{
		C(X,\mathcal{A}) \ar[r] \ar[d] & C(V_0,\left.\mathcal{A}\right|_{V_0}) \ar[d] \\
		C(V_1, \left.\mathcal{A}\right|_{V_1}) \ar[r] & C(V_0 \cap V_1, \left.\mathcal{A}\right|_{V_0 \cap V_1})
	}
\]
by \cite[Thm.\ 21.2.3]{book:Blackadar}. That excision, i.e.\ (\ref{it:ex}), holds for the twisted $K$-functor follows from the contractibility of the algebra $C_0(U \times [0,1), \left.c\,\mathcal{A}\right|_{U \times [0,1)})$ and the long exact sequence associated to 
	\[
		0 \to C_0(U \times [0,1), \left.c\,\mathcal{A}\right|_{U \times [0,1)}) \to C(c(X,B), c\,\mathcal{A}) \to C(c(X\!\setminus\! U, B\!\setminus\! U), \left.c\,\mathcal{A}\right|_{c(X\setminus U, B \setminus U)}) \to 0 \qedhere
	\]
\end{proof}

\begin{remark} \label{rem:module_structure}
It is straightforward to check that the $K^{\bullet}_D(X,B)$-module structure agress with the one obtained from the $KU^D_{\bullet}$-module structure of the spectrum $\KUmod{D}{\bullet}$.
\end{remark}
	
\begin{remark} \label{rem:homotopy_fibration}
There is an alternative way to prove the six-term exact sequence from Thm.\ \ref{thm:Twisted_K}~(\ref{it:d}), which is closer to stable homotopy theory: Let $K_n^{(X,B)} = K_n$ be the space from the proof of Thm.~\ref{thm:Twisted_K} and let $K_n^{X} := K_n^{(X,\emptyset)}$. One can show that the homotopy fiber of the natural restriction map $K_n^{(X,B)} \to K_n^{X}$ is weakly homotopy equivalent to $\Omega K_n^B$. Moreover, the homotopy fiber of $K_n^X \to K_n^B$ is $K_n^{(X,B)}$. Thus, we can look at the fibration sequence
\[
	\xymatrix{
		& & \Omega K_n^B \ar[r] & K_n^{(X,B)} \ar[r] & K_n^X \ar[r] & K_n^B \\
		K_n^{(X,B)} \ar[r] & K_{n-1}^{X} \ar[r] & K_{n-1}^{B} \ar[u]
	}
\]
in which the vertical map is a weak equivalence for $n > 1$. The exact sequence in (\ref{it:d}) now follows from Bott periodicity and the long exact sequence of homotopy groups.
\end{remark}

\subsection{Parametrized stable homotopy theory}
A slightly different approach to twisted $K$-Theory is based on parametrized stable homotopy theory \cite{book:MaySigurdsson}. To compare it with the definition given above, we  will develop an orthogonal version of the bundle of symmetric $KU_{\bullet}$-module spectra. It has the advantage that it yields parametrized spectra over $B\Aut{A}$ for arbitrary $C^*$-algebras $A$. The price one has to pay for this is that its fibers are only $KU_{\bullet} = KU_{\bullet}^{\C}$-modules in a natural way. Nevertheless, in case $A = \Cuntz{\infty} \otimes \K$, the space $B\Aut{\Cuntz{\infty} \otimes \K}$ has the homotopy type of $BGL_1(KU)$ \cite[Thm.\ 4.6]{paper:DadarlatP2} and the resulting theory is an extension of the above symmetric version representing twisted $K$-Theory including all higher twists. 

\begin{definition} \label{def:orth_KU}
Let $A$ be a $\sigma$-unital $C^*$-algebra. Given a finite dimensional inner product space $V$, we define 
\[
	\KUort{A}{V} = \hom_{\rm gr}(\Shat, A \otimes \C\ell(V) \otimes \K(L^2(V)))\ ,
\]
which is pointed by the zero homomorphism. 
\end{definition}

We denote $\KUort{\C}{V}$ by $KU(V)$. The isometric isomorphism $L^2(V \oplus W) \cong L^2(V) \otimes L^2(W)$ induces an isomorphism of $C^*$-algebras $\K(L^2(V)) \otimes \K(L^2(W)) \to \K(L^2(V \oplus W))$. Moreover, we have an isomorphism $\C\ell(V) \otimes \C\ell(W) \cong \C\ell(V \oplus W)$ of graded $C^*$-algebras. Just as in Def.~\ref{def:KUD_spectrum} we therefore obtain a multiplication map 
\[
	\mu_{V,W} \colon KU(V) \wedge KU(W) \to KU(V \oplus W)\ .
\]
Let $S^V$ be the one-point compactification of $V$. The canonical linear map $V \to \C\ell(V)$ defines an unbounded, self-adjoint, regular, odd multiplier $C_0(V, \C\ell(V))$, which induces the map $\eta_V \colon S^V \to \hom_{\rm gr}(\Shat, \C\ell(V) \otimes \K(L^2(V)))$. It was shown in \cite{paper:JoachimEquiv} that $(KU(\bullet), \mu_{\bullet, \bullet}, \eta_{\bullet})$ is an orthogonal ring spectrum. 

As in the symmetric case there is an action $\alpha_{V,W} \colon \KUort{A}{V} \wedge KU(W) \to \KUort{A}{V \oplus W}$ of $KU$ on $KU^A$ given by $\varphi \wedge \psi \mapsto (\varphi \otimes \psi) \circ \Delta$, where we suppress the above isomorphisms in the notation. Apart from the action of $KU$, $\KUort{A}{V}$ also carries a (left) action of the group $\Aut{A}$. We will omit the proof of the following lemma, which is completely analogous to that of Thm.\ \ref{thm:KU_module}.

\begin{lemma}
The pair $(\KUort{A}{\bullet}, \alpha_{\bullet, \bullet})$ forms an orthogonal module spectrum over $KU$. The module structure is compatible with the action of $\Aut{A}$ described above. The homotopy groups of $\KUort{A}{\bullet}$ are given by $\pi_i(\KUort{A}{\bullet}) = K_i(A)$.
\end{lemma}

We will consider the category of spaces $X$ over a fixed base space $B$, such that each fiber is continuously pointed. These are called ex-spaces.

\begin{definition} \label{def:ex-space}
An \emph{ex-space} over a base space $B$ is a space $X$ together with maps $\pi \colon X \to B$ and $\sigma \colon B \to X$ such that $\pi \circ \sigma = \id{B}$. 
\end{definition}

There is a good notion of smash product $X \wedge_B Y$ for ex-spaces defined by forming smash products fiberwise, and there is also an internal Hom-functor $F_B(X,Y)$ adjoint to $\wedge_B$. It associates to two ex-spaces $X$ and $Y$ over $B$ another ex-space $F_B(X,Y)$, such that the fiber over $b \in B$ agrees with continuous based maps $X_b \to Y_b$ \cite[Def.\ 1.3.12]{book:MaySigurdsson}. Let $f \colon A \to B$ be a continuous map. As explained in \cite[Sec.\ 2.1]{book:MaySigurdsson} there is a pullback functor $f^*$, which has a right adjoint $f_*$ and a left adjoint $f_!$. Let $r \colon B \to {\ast}$ be the canonical map. The parametrized loop space functor is given by $\Omega_BX = F_B(r^*S^1, X)$. Any space $X$ over $B$ can be turned into an ex-space by adding a disjoint basepoint in each fiber, i.e.\ by setting $X_{+} = B \amalg X$ with the obvious projection and section. Given a real finite dimensional inner product space $V$ let 
\[
	\KUobun{A}{V} = E\Aut{A} \times_{\lambda} \KUort{A}{V}\ .
\]
Together with the projection $\pi \colon \KUobun{A}{V} \to B\Aut{A}$ and zero section $\sigma \colon B\Aut{A} \to \KUobun{A}{V}$, this is an ex-space over $B\Aut{A}$. Since $\alpha_{V,W}$ is equivariant with respect to the action of $\Aut{A}$, it extends to a map
\[
	\bar{\alpha}_{V,W} \colon \KUobun{A}{V} \wedge_{B\Aut{A}} KU(W) \to \KUobun{A}{V \oplus W}
\]
The following corollary is the outcome of \cite[Sec.\ 20.2]{book:MaySigurdsson}.

\begin{corollary}
The family $(\KUobun{A}{\bullet}, \bar{\alpha}_{\bullet,\bullet})$ is a parametrized $KU$-module spectrum over $B\Aut{A}$ in the sense of \cite[Def.\ 14.1.1]{book:MaySigurdsson}. Given an ex-space $(X,f,\sigma)$ over $B = B\Aut{A}$, let $\Pb = f^*E\Aut{A}$ and
\[
	\widetilde{K}^{i, {\rm para}}_{\Pb}(X) = \pi_{i}(r_*F_{B}(X, \KUobun{A}{\bullet}))\ .
\]
This is a (reduced) parametrized cohomology theory in the sense of \cite[Def.\ 20.1.2]{book:MaySigurdsson}.
\end{corollary}

\begin{theorem} \label{thm:symm_orth_equiv}
Let $D$ be a strongly self-absorbing $C^*$-algebra, let $A = D \otimes \K$. 
\begin{enumerate}[(a)]
	\item \label{it:rs} If $D \in \{ \Cuntz{\infty}, \JiangSu, \C \}$, then the symmetric ring spectrum obtained from $n \mapsto KU(\R^n)$ by restricting the $O(n)$-action to a $\Sigma_n$-action is equivalent to $KU^D_{\bullet}$ as a ring spectrum.
	\item \label{it:ms} The symmetric spectrum obtained from $\KUort{A}{\bullet}$ by restricting the $O(n)$-action to a $\Sigma_n$-action is equivalent to $\KUmod{D}{\bullet}$ as a $KU^{\C}_{\bullet}$-module spectrum and the equivalence commutes with the group action of $\Aut{D \otimes \K}$ on both sides.
	\item \label{it:kt} Let $(X,B,f) \in \TopPair{D}$, let $\Pb = f^*E\Aut{A}$, let $\iota_+ \colon B_+ \to X_+$ be the induced inclusion of ex-spaces and denote by $c(\iota_+)$ the  (parametrized) mapping cone of $\iota_+$. The parametrized twisted $K$-groups $K^{i, {\rm para}}_{\Pb}(X,B) = \widetilde{K}^{i, {\rm para}}_{\Pb}(c(\iota_+))$ are naturally isomorphic to the twisted $K$-groups from Def.\ \ref{def:twisted_K}, i.e.\ $K^{i, {\rm para}}_{\Pb}(X,B) \cong K^i_{\Pb}(X,B)$.
\end{enumerate}
\end{theorem}

\begin{proof} \label{pf:symm_orth_equiv}
Observe that $\C\ell(\R^n) \otimes \K(L^2(\R^n)) \cong (\C\ell_1 \otimes \K(L^2(\R)))^{\otimes n} = (\C\ell_1 \otimes \K)^{\otimes n}$, where the isomorphism can be chosen to be compatible with the action of $\Sigma_n$ on both sides. Let $\kappa \colon \C\ell_1 \otimes \K \to \C\ell_1 \otimes D \otimes \K$ be the map that sends $v \otimes T$ to $v \otimes 1_D \otimes T$. This induces a natural map $KU(\R^n) \to KU^D_n$, which sends $\varphi \colon \Shat \to (\C\ell_1 \otimes \K)^{\otimes n}$ to the homomorphism $\hat{\varphi} \colon \Shat \to (\C\ell_1 \otimes D \otimes \K)^{\otimes n}$ with $\hat{\varphi} = \kappa^{\otimes n} \circ \varphi$. On homotopy groups it induces the homomorphism $K_*(\C) \to K_*(D)$ given by $[p] \mapsto [p \otimes 1_D]$. This is an isomorphism for $D \in \{\Cuntz{\infty}, \JiangSu, \C\}$. Since both sides are semistable spectra, they are equivalent. We have 
\[
	(\hat{\varphi} \otimes \hat{\psi}) \circ \Delta = (\kappa^{\otimes n} \otimes \kappa^{\otimes m}) \circ (\varphi \otimes \psi) \circ \Delta = \kappa^{\otimes n+m} \circ (\varphi \otimes \psi) \circ \Delta\ .
\]
Moreover, composition of the map $\eta_n^{KU} \colon S^n \to KU(\R^n)$ with the above map $KU(\R^n) \to KU^D_n$ yields  $\eta_n^{KU^D}$, i.e.\ $\varphi \mapsto \hat{\varphi}$ preserves units. This proves (\ref{it:rs}).

The proof of (\ref{it:ms}) is similar: Again we identify $\C\ell(\R^n) \otimes K(L^2(\R^n))$ with $(\C\ell_1 \otimes \K)^{\otimes n}$. The map $\KUort{A}{\R^n} \to \KUmod{D}{n}$ sends $\varphi$ to $\tilde{\varphi} = (\id{D \otimes \K} \otimes \kappa^{\otimes n}) \circ \varphi$, which induces an isomorphism on homotopy groups and therefore provides a stable equivalence. The map is compatible with the left action of $\Aut{D \otimes \K}$. The following diagram commutes
\[
	\xymatrix@C=1.5cm@R=0.8cm{
		\KUort{A}{\R^n} \wedge KU(\R^m) \ar[r] \ar[d]_-{\alpha_{\R^n, \R^m}} & \KUmod{D}{n} \wedge KU^{\C}_m \ar[d]^-{\alpha^{\C}_{n,m}}\\
		\KUort{A}{\R^{n+m}} \ar[r] & \KUmod{D}{n+m}
	}
\]
where $\alpha_{n,m}^{\C}$ is induced by the map of spectra $KU^{\C}_{\bullet} \to KU^D_{\bullet}$. This proves that the stable equivalence intertwines the $KU$-action on the left hand side with the one by $KU^{\C}_{\bullet}$ on the right.

The mapping cone $c(\iota_+)$ is given by $B_+ \wedge_{B\Aut{A}} I \cup_{\iota_+} X_+$, where $I$ has basepoint $0$. This space is homeomorphic as an ex-space to a quotient $(M(\iota)_+)/\!\!\sim$ of the mapping cylinder, where the equivalence relation identifies all points $(b,0) \in M(\iota)$ in a fiber of $\left.f\right|_B$ with the disjoint basepoint in that fiber. Let $g \colon Y \to C$ be a space over $C = B\Aut{A}$. By local triviality of the bundle $\KUobun{A}{\R^n}$ and \cite[Prop.~2.2.2, eqn.\ (2.2.7)]{book:MaySigurdsson} the mapping space $F_C(Y_+, \KUobun{A}{\R^n})$ is homeomorphic to ${\rm Sec}(Y, g^*\KUbun{A}{\R^n})$. It follows from this that $F_C(c(\iota_+), \KUobun{A}{\R^n})$ is homeomorphic to the subspace of ${\rm Sec}(M(\iota), f^*\KUobun{A}{\R^n})$ consisting of sections $\tau \colon M(\iota) \to f^*\KUobun{A}{\R^n}$ with $\tau(b,0) = \sigma_n(f(b))$, where $\sigma_n \colon B\Aut{A} \to \KUobun{A}{\R^n}$ is the zero section. For $(X,B) \in \TopPair{D}$ this space is homeomorphic to 
\[
	\widehat{K}_n = \hom_{\rm gr}(\Shat, C_0(c(X,B), c\,\mathcal{A}) \otimes (\C\ell_1 \otimes \K)^{\otimes n})
\]
Let $K_n$ be the space from the proof of Thm.\ \ref{thm:Twisted_K}. The map of spectra from (\ref{it:ms}) translates into a map $\widehat{K}_n \to K_n$, which sends $\varphi$ to $(\id{} \otimes \kappa^{\otimes n}) \circ \varphi$. Using the same identification of the homotopy groups with the $K$-groups as in the proof of Thm.\ \ref{thm:Twisted_K} we have a commutative diagram
\[
	\xymatrix@R=0.6cm{
		\pi_{i+n}(\widehat{K}_n) \ar[rr]^-{\varphi \mapsto (\id{} \otimes \kappa^{\otimes n}) \circ \varphi} \ar[dr]_-{\cong} & & \pi_{i+n}(K_n) \ar[dl]^-{\cong} \\
		& K_i(C_0(c(X,B), c\,\mathcal{A}))
	}
\]
proving (\ref{it:kt}).
\end{proof}

\section{Twisted $K$-homology}
For twists classified by $H^3(X,\Z)$ twisted $K$-homology is naturally isomorphic to analytic $K$-homology of a corresponding bundle of compact operators \cite{paper:Wang1}. In this section, we will extend these results to include the higher twists as well. We will start by recalling the definition of topological twisted $K$-homology from \cite{paper:Ando, paper:ABGH}, which uses $\infty$-categories. We will also express twisted $K$-homology in terms of $KK$-theory. This is based on a Mayer-Vietoris argument and will not require any knowledge of $\infty$-categories.

\subsection{Generalized Thom-spectra}
Let $R$ be a ring spectrum. In \cite{paper:Ando, paper:ABGH} the authors apply the theory of  $\infty$-categories to associate a Thom spectrum $Mf$ to a map $f \colon X \to BGL_1(R)$, which they use to define twisted versions of $R$-cohomology and $R$-homology. 

An $\infty$-category is by definition a simplicial set, which satisfies the weak Kan condition. They were called quasicategories by Joyal. Let $\Sp$ be the $\infty$-category of spaces (or $\infty$-groupoids) and denote by $\St{\Sp}$ its stabilization, i.e.\ the $\infty$-category of spectra obtained by taking spectrum objects in $\Sp$ as described in \cite[Sec.\ 1.4.2]{book:Lurie2}. Denote by $\RMod{R} \subset \St{\Sp}$ the full subcategory of $\St{\Sp}$ of $R$-module spectra. Let $\RLine{R} \subset \RMod{R}$ be the subcategory of free rank one $R$-modules and equivalences of those. Let $\mathcal{B}{\rm Aut}_R(R) \subset \RLine{R}$ be the full subcategory on the object $R$. The latter are both $\infty$-groupoids and it was shown in \cite[Cor.\ 2.9]{paper:ABGH} that $\mathcal{B}{\rm Aut}_R(R) \simeq \RLine{R}$ and both have the homotopy type $BGL_1(R)$. 

To give an $\infty$-categorical model for $\RLine{KU^D}$ based on $C^*$-algebras we will use the following technique to turn topological categories into $\infty$-categories: Given a category $\mathcal{C}$ enriched in topological spaces, we obtain an associated simplicially enriched category $\mathcal{C}^s$ with ${\rm Obj}(\mathcal{C}^s) = {\rm Obj}(\mathcal{C})$ and ${\rm Mor}_{\mathcal{C}^s}(x,y) = {\rm Sing}({\rm Mor}_{\mathcal{C}}(x,y))$, where ${\rm Sing}(\,\cdot\,)$ is the singular simplicial set. Let $N_{hc} \colon {\rm sCat} \to {\rm sSet}$ be the homotopy coherent nerve functor as defined by Cordier \cite{paper:Cordier}, \cite[]{book:Lurie}. The simplicial sets ${\rm Mor}_{\mathcal{C}^s}(x,y)$ are Kan complexes. Therefore $\mathcal{C}^{\infty} = N_{hc}(\mathcal{C}^s)$ is an $\infty$-category. Let $\mathcal{B}_D$ be the category with one object and morphism space $\Aut{D \otimes \K}$ and let $\mathcal{B}_D^{\infty}$ be the associated $\infty$-category. We have a functor $\mathcal{B}_D^{\infty} \to \mathcal{B}{\rm Aut}_{KU^D}(\KUmod{D}{\bullet}) \subset \RLine{KU^D}$ that sends the object to $\KUmod{D}{\bullet}$ and maps $\alpha \in \Aut{D \otimes \K}$ to the induced map of module spectra as above. On geometric realizations this map agrees with the one from \cite[Thm.\ 4.6]{paper:DadarlatP1}, i.e.\ it induces an isomorphism on $\pi_k$ for $k > 0$ and the inclusion $K_0(D)^{\times}_+ \to K_0(D)^{\times}$ on $\pi_0$.

\begin{definition} \label{def:twisted_K_homology}
Let $X$ be a finite CW-complex, let $\Pb \to X$ be a principal $\Aut{D \otimes \K}$-bundle classified by $f \colon X \to B\Aut{D \otimes \K}$ and let $f^*\KUbun{D}{\bullet} \to X$ be the associated bundle of symmetric $KU^D$-module spectra. Let $\sigma_n \colon X \to f^*\KUbun{D}{n}$ be the zero section. The \emph{Thom spectrum} of $f$ is defined by
\[
	Mf_n = f^*\KUbun{D}{n}/\sigma_n(X) = r_!f^*\KUbun{D}{n}\ .
\]
It is a symmetric module spectrum over $KU^D_{\bullet}$ with respect to the action maps induced by $\alpha_{\bullet,\bullet}$. Define the topological twisted $K$-homology of $X$ as 
\[
	K^{\Pb, {\rm top}}_i(X) = \pi_i(Mf_{\bullet}) = \pi_i(r_!f^*\KUbun{D}{\bullet})
\]
where $r \colon X \to \{\ast\}$. 
\end{definition}

Let $G = \Aut{D \otimes \K}$. The spectrum $Mf_{\bullet}$ in Def.\ \ref{def:twisted_K_homology} agrees with the smash product $\Sigma^{\infty}P_+ \wedge_{\Sigma^{\infty}G_+} \KUmod{D}{\bullet}$ in symmetric spectra. Since we assumed $X$ to be a finite CW complex, $P$ has the structure of a (free) $G$-CW-complex. In particular, $\Sigma^{\infty}P_+$ turns out to be a cofibrant $\Sigma^{\infty} G_+$-module spectrum. Therefore the above is equivalent to the derived smash product, which was called the algebraic Thom spectrum in \cite{paper:ABGH}. It was proven in \cite[Prop.\ 3.26]{paper:ABGH} that this notion is equivalent in the stable $\infty$-category $KU^D$-Mod to the geometric Thom spectrum \cite[Def.\ 1.4]{paper:ABGH}. Together with \cite[Rem.\ 5.4]{paper:Ando} we obtain the following:
\begin{theorem} \label{thm:Twisted_K_via_inf_cat}
Let $X$ be a finite CW-complex, let $\Pb \to X$ be a principal $\Aut{D \otimes \K}$-bundle, let $\mathcal{A}$ be the associated $D \otimes \K$-bundle. Let $\Pb^{-} \to X$ be a principal $\Aut{D \otimes \K}$-bundle, such that the $D \otimes \K$-bundle $\mathcal{A}^-$ associated to it satisfies $C(X,\mathcal{A} \otimes \mathcal{A}^-) \cong C(X, D \otimes \K)$. Then we have 
\[
	K^i_{\Pb^-}(X) \cong K_i(C(X, \mathcal{A}^-)) \cong \pi_0({\rm Map}_{KU^D}(Mf_{\bullet}, \Sigma^i \KUmod{D}{\bullet} )\ . 
\]
\end{theorem}

Given a CW-complex $X$, a principal $\Aut{D \otimes \K}$-bundle $\Pb \to X$ classified by a map $f \colon X \to B\Aut{D \otimes \K}$ and a subcomplex $A \subset X$, there is a canonical map between the associated Thom spectra $M(\left.f\right|_A)_{\bullet} \to Mf_{\bullet}$, which is obtained from the inclusion $(\left.f\right|_A)^*\KUbun{D}{\bullet} \to f^*\KUbun{D}{\bullet}$ covering the map $A \to X$.

\begin{theorem} \label{thm:Mayer-Vietoris}
Let $(X,A,B)$ be an excisive CW-triad. Let $\Pb$ be a principal $\Aut{D \otimes \K}$-bundle over $X$ and denote its restrictions to $A$, $B$ and $A \cap B$ by $\Pb_{A}$, $\Pb_{B}$ and $\Pb_{AB}$ respectively. There is a boundary homomorphism $\partial \colon K_i^{\Pb, {\rm top}}(X) \to K_{i-1}^{\Pb_{AB}, {\rm top}}(A \cap B)$ such that the Mayer-Vietoris sequence of twisted $K$-homology groups is exact:
\[
	\xymatrix@C=0.5cm{
		\dots \ar[r] & K_i^{\Pb_{AB}, {\rm top}}(A \cap B) \ar[r] & K_i^{\Pb_{A}, {\rm top}}(A) \oplus K_i^{\Pb_{B}, {\rm top}}(B) \ar[r] & K_i^{\Pb, {\rm top}}(X) \ar[r]^-{\partial} & K_{i-1}^{\Pb_{AB}, {\rm top}}(A \cap B) \ar[r] & \dots
	}
\] 
\end{theorem}

\begin{proof} \label{pf:Mayer-Vietoris}
Let $I$ be the pushout diagram category and let $\mathcal{S}_{/X}$ be the $\infty$-category of spaces over $X$. The pushout diagram 
$
	\xymatrix{
		A & \ar[l] A \cap B \ar[r] & B
	}
$
defines a functor $F \colon I \to \mathcal{S}_{/X}$, since each of the spaces comes equipped with a map to $X$. Let $f \colon X \to KU^D{\rm -Line} \to KU^D{\rm -Mod}$ be the diagram in the $\infty$-categorical definition of $Mf_{\bullet}$. The colimits of $F(i) \to X \to KU^D{\rm -Mod}$ correspond to $M\!\left(\left.f\right|_{A \cap B}\right)_{\bullet}$, $M\!\left(\left.f\right|_A\right)_{\bullet}$ and $M\!\left(\left.f\right|_B\right)_{\bullet}$ respectively. The category $I$ is partially ordered in such a way that $F$ is an order-preserving map from $I$ to the collection of simplicial subsets of $X$. Using \cite[Rem.\ 4.2.3.9]{book:Lurie} we see that the conditions of \cite[Prop.\ 4.2.3.8]{book:Lurie} are satisfied. Thus, \cite[Cor.\ 4.2.3.10]{book:Lurie} yields that 
\[
	\xymatrix{
		M\!\left(\left.f\right|_{A \cap B}\right)_{\bullet} \ar[r] \ar[d] & M\!\left(\left.f\right|_{A}\right)_{\bullet} \ar[d] \\
		M\!\left(\left.f\right|_{B}\right)_{\bullet} \ar[r] & Mf_{\bullet}
	}
\]
is an $\infty$-categorical pushout diagram in $KU^D$-Mod. By stability of the latter, this is also a pullback diagram in $KU^D$-Mod. The associated long exact sequence of homotopy groups is the stated Mayer-Vietoris sequence.
\end{proof}

\subsection{Analytic twisted $K$-homology}
Let $X$ be a locally compact space, let $\Pb \to X$ be a principal $\Aut{D \otimes \K}$-bundle and denote the associated bundle of $C^*$-algebras by $\mathcal{A}$. We will use Kasparov's representable bivariant $K$-theory, which was defined in \cite{paper:Kasparov}. Analogous to the twisted $K$-theory functor, let
\begin{align*}
RK^*_{\Pb}(X) &= \RKK{X}{C_0(X) \otimes D}{C_0(X, \mathcal{A})}{\ast} \\
RK_*^{\Pb}(X) &= \RKK{X}{C_0(X,\mathcal{A})}{C_0(X) \otimes D}{-\ast} \\
K_*^{\Pb}(X, Y) &= KK_{-*}(C_0(X \setminus Y,\left.\mathcal{A}\right|_{X \setminus Y}), D) 
\end{align*}
To compare analytic with topological twisted $K$-homology, we need a version of Poincar\'{e} duality. Let $M$ be a smooth compact spin$^c$ manifold with (possibly empty) boundary $\partial M$. Let $\Pb \to M$ be a principal $\Aut{D \otimes \K}$-bundle with associated bundle of $C^*$-algebras $\mathcal{A}$. Let $M^{\circ} = M \setminus \partial M$. The spin$^c$-condition implies that $M$ has a fundamental class $[M, \partial M] \in K_{\dim(M)}(M, \partial M)$. For a compact space $X$ we have an isomorphism $K^*_{\Pb}(X) \cong RK^*_{\Pb}(X)$. Given two continuous $C_0(X)$-algebras $A$ and $B$ the group $\RKK{X}{A}{B}{*}$ maps naturally to $KK_*(A,B)$ by forgetting the additional assumptions on the cycles in representable $KK$-theory. It was shown in \cite[Thm.\ 3.8]{paper:DadarlatP1} that the isomorphism classes of locally trivial principal $\Aut{D \otimes \K}$-bundles form an abelian group with respect to the tensor product. Thus, there is a bundle $\mathcal{A}^{-} \to M$ with the property that $C(M,\mathcal{A}) \otimes_{C(M)} C(M, \mathcal{A}^{-}) \cong C(M, D \otimes \K)$. This yields a class $[\mathcal{A}^{-}] \in \RKK{M}{C(M,\mathcal{A}^{-})}{C(M,\mathcal{A}^{-})}{0}$ and via $x \mapsto x \otimes [\mathcal{A}^{-}]$ an isomorphism
\[
	RK^*_{\Pb}(M) = \RKK{M}{C(M) \otimes D}{C(M,\mathcal{A})}{\ast} \to \RKK{M}{C(M, \mathcal{A}^{-})}{C(M) \otimes D}{\ast}
\]
Let $i \colon M^{\circ} \to M$ be the inclusion and denote by $i^*$ the induced homomorphism on $\mathcal{R}KK_*$. By \cite[Prop.\ 4.9]{paper:EEK} the homomorphism 
\[
	i^* \colon \RKK{M}{C(M, \mathcal{A}^{-})}{C(M) \otimes D}{\ast} \to \RKK{M^{\circ}}{C_0(M^\circ, \mathcal{A}^{-})}{C_0(M^{\circ}) \otimes D}{\ast}
\]
is an isomorphism. Combining the above observations with the map to $KK$ we obtain $\iota_{(M,\partial M)}$:
\[
\xymatrix{
	 K^*_{\Pb}(M) \ar[r]^-{\cong} & \RKK{M^{\circ}}{C_0(M^{\circ}, \mathcal{A}^{-})}{C_0(M^{\circ}) \otimes D}{\ast} \ar[r] & KK_*(C_0(M^{\circ},\mathcal{A}^{-}), C_0(M^{\circ}) \otimes D)
}\ .
\]
We define the Poincar\'{e} duality homomorphism by mapping the last group to $K_{\dim(M) - *}^{\Pb^{-}}(M, \partial M)$:  
\begin{equation} \label{eqn:PD}
	PD_{*} \colon K^*_{\Pb}(M) \to K_{\dim(M) - *}^{\Pb^{-}}(M, \partial M) \quad ; \quad y \mapsto \iota_{(M,\partial M)}(y) \otimes [M, \partial M]\ .
\end{equation}
The boundary homomorphism $\partial \colon K^{\Pb^{-}}_{\dim{M} - *}(M, \partial M) \to K^{\Pb^{-}}_{\dim(\partial M) - *}(\partial M)$ is given by an intersection product with the class $[\partial] \in KK_1(C(\partial M, \left.\mathcal{A}^{-}\right|_{\partial M}), C_0(M^{\circ},\mathcal{A}^{-}))$, which represents the short exact sequence
\[
	0 \to C_0(M^{\circ}, \mathcal{A}^{-}) \to C(M, \mathcal{A}^{-}) \to C(\partial M, \left.\mathcal{A}^{-}\right|_{\partial M}) \to 0
\]

\begin{lemma} \label{lem:partial1}
Let $j \colon \partial M \to M$ be the inclusion, then we have 
\[
	\partial \otimes \iota_{(M, \partial M)}(x) = (-1)^{\deg(x)}\,\iota_{\partial M}(j^*x) \otimes \partial\ .
\]
\end{lemma}

\begin{proof}
There is a second way to obtain the map $\iota_{(M,\partial M)}$, namely via
\[
\xymatrix@C=0.8cm{
	K^*_{\Pb}(M) \ar[r]^-{\otimes [\mathcal{A}^{-}]}& KK_*(C_0(M^{\circ}, \mathcal{A}^-), C(M, \mathcal{A}) \otimes C_0(M^{\circ}, \mathcal{A}^{-})) \ar[r]^-{\mu} & KK_*(C_0(M, \mathcal{A}^{-}), C_0(M^{\circ}) \otimes D)
}
\]
where $\mu$ is induced by the diagonal inclusion $M^{\circ} \to M \times M^{\circ}$ and the trivializing isomorphism $C_0(M^{\circ}, \mathcal{A} \otimes \mathcal{A}^{-}) \cong C_0(M^{\circ}, D \otimes \K)$. From here the proof proceeds just as in \cite[Lem.\ B.8]{paper:BOSW}
\end{proof}

\begin{corollary} \label{cor:partial2}
Let $j \colon \partial M \to M$ be the inclusion, then the following diagram commutes up to a sign, which is given by $(-1)^{\deg(x)}$ for $x \in K^*_{\Pb}(M)$
\[
	\xymatrix@C=2cm{
		K^*_{\Pb}(M) \ar[r]^-{j^*} \ar[d]_-{PD_*} & K^*_{\Pb}(\partial M) \ar[d]^-{PD_*}\\
		K_{\dim(M) - *}^{\Pb^{-}}(M, \partial M) \ar[r]_-{\partial} & K_{\dim(\partial M) - *}^{\Pb^-}(\partial M)
	}
\]
\end{corollary}

\begin{proof} \label{pf:partial2}
This is a direct consequence of Lemma~\ref{lem:partial1} and \cite[Lem.\ B.10]{paper:BOSW}.
\end{proof}

Fix a compact Hausdorff space $X$ together with a principal $\Aut{D \otimes \K}$-bundle $\Pb \to X$ pulled back via $f \colon X \to B\Aut{D \otimes \K}$. Let $\widehat{\K}$ be the graded compact operators on a graded separable infinite-dimensional Hilbert space, let 
\[
	\Kbun{D} = E\Aut{D \otimes \K} \times_{\lambda} \hom_{\rm gr}(\Shat, D \otimes \K \otimes \widehat{\K})\ ,
\]
where $\lambda$ denotes the left action of $\Aut{D \otimes \K}$ on $D \otimes \K$. Let $M$ be a closed spin$^c$-manifold and let $g \colon M \to f^*\Kbun{D}$ be a continuous map to the total space of the bundle. Let $\bar{g} \colon M \to X$ be the map induced by $g$ and the bundle projection. Observe that $g$ yields a class $[E_g] \in K^0_{\bar{g}^*\Pb}(M)$. There is a homotopy associative multiplication map $m \colon f^*\Kbun{D} \wedge \hom_{\rm gr}(\Shat, \widehat{\K}) \to f^*\Kbun{D}$, which uses an isomorphism $\widehat{\K} \otimes \widehat{\K} \cong \widehat{\K}$ of graded $C^*$-algebras. In particular, the Bott element induces a weak equivalence $\Kbun{D} \to \Omega^2\Kbun{D}$ (where the loop space is taken fiberwise). The bundle $\Kbun{D}$ is the twisted analogue of the space $X \times \K$ in \cite[Sec.\ 8]{paper:BaumHigsonSchick}.

\begin{definition} \label{def:ana_ind}
For $(M,g)$ as described above, we define the \emph{analytic index} of the pair by ${\rm ind}_a(M,g) = \bar{g}_* PD_0([E_g]) \in K_{\dim(M)}^{\Pb^{-}}(X)$.
\end{definition}

\begin{theorem} \label{thm:bordism_invariance}
The analytic index only depends on the (spin$^c$) bordism class $[M,g] \in \Omega^{{\rm spin}^c}(f^*\Kbun{D})$.
\end{theorem}

\begin{proof} \label{pf:bordism_invariance}
Let $(W, h)$ be a spin$^c$ manifold with boundary $\partial W = M_1 \amalg -M_2$ together with a continuous map $h \colon W \to f^*\Kbun{D}$, which restricts to $g_i \colon M_i \to f^*\Kbun{D}$. Then we have $PD_0([E_h]) \in K_{\dim(W)}^{\bar{h}^*\Pb^{-}}(W, \partial W)$. The class $(\left.\bar{h}\right|_{\partial W})_* \partial(PD_0([E_h]))$ vanishes. This follows from the naturality of the boundary map by the commutativity of the following diagram:
\[
	\xymatrix{
		K_{\dim(W)}^{\bar{h}^*\Pb^{-}}(W, \partial W) \ar[d]_-{\bar{h}_*} \ar[r]^-{\partial} & K_{\dim(\partial W)}^{\bar{h}^*\Pb^{-}}(\partial W) \ar[d]^-{(\left.\bar{h}\right|_{\partial W})_*}\\
		0 = K_{\dim(W)}^{\Pb^{-}}(X, X) \ar[r]_-{\partial} & K_{\dim(\partial W)}^{\Pb^{-}}(X)
	}
\]
By Cor.\ \ref{cor:partial2}, we have $0 = (\left.\bar{h}\right|_{\partial W})_* \partial(PD_0([E_h])) = (\left.\bar{h}\right|_{\partial W})_* PD_0(j^*([E_h])) = {\rm ind}_a(M_1,g_1) - {\rm ind}_a(M_2,g_2)$, which proves the statement.
\end{proof}

\subsection{Comparison of topological and analytic twisted $K$-homology}
In this section we prove that the topological twisted $K$-homology of a finite CW-complex $X$ is isomorphic to its analytic counterpart. To achieve this we extend the idea of \cite{paper:BaumHigsonSchick} to the twisted case, i.e.\ we construct an intermediate twisted homology theory, which maps to the topological and the analytic one. 

As we have seen in the definition of the analytic index map above, the twist changes to its inverse under Poincar\'{e} duality. Since this will play a central role in the natural transformations between the theories we are going to construct, we need a functorial way of inverting twists. This is not a problem in \cite{paper:ABGH, paper:Ando}, where it suffices to have an endofunctor of the $\infty$-category $KU^D$-Line. Instead of rephrasing the whole theory in that framework, we simply build the inverse into the category of spaces under consideration. 

Let $\Pb$ and $\Pb'$ be principal $\Aut{D \otimes \K}$-bundles over the same base space $X$. The tensor product of automorphisms induces a group homomorphism $\kappa \colon \Aut{D \otimes \K} \times \Aut{D \otimes \K} \to \Aut{(D \otimes \K)^{\otimes 2}}$ and we define
\[
	\Pb \otimes \Pb' = (\Pb \times_M \Pb') \times_{\kappa} \Aut{(D \otimes \K)^{\otimes 2}}
\] 
If $\mathcal{A}$ and $\mathcal{A}'$ are the corresponding $C^*$-algebra bundles, then $\mathcal{A} \otimes \mathcal{A}'$ is associated to $\Pb \otimes \Pb'$.

\begin{definition} \label{def:cat_CWD}
Let $\CW{D}$ be the following category: The objects are tuples $(X,f,f^-, \tau)$, where $X$ is a finite CW-complex, $f, f^- \colon X \to B\Aut{D \otimes \K}$ are continuous maps classifying bundles $\Pb_f = f^*E\Aut{D \otimes \K}$ and $\Pb_f^- = (f^-)^*E\Aut{D \otimes \K}$ and $\tau \colon \Pb_f \otimes \Pb_f^- \to X \times \Aut{(D \otimes \K)^{\otimes 2}}$ is a trivialization. A morphism $(X,f,f^-,\tau) \to (X', f', f^{'-}, \tau')$ is tuple $(\varphi, \hat{\varphi}, \hat{\varphi}^-, \rho)$, where $\varphi \colon X \to X'$ is a continuous map and $\hat{\varphi} \colon \Pb_f \to \Pb_{f'}$, $\hat{\varphi}^- \colon \Pb_{f}^- \to \Pb_{f'}^-$ are maps of principal $\Aut{D \otimes \K}$-bundles covering $\varphi$ and $\rho \colon X \to \Aut{(D \otimes \K)^{\otimes 2}}$ is a change of trivialization, which fits into the commutative diagram
\[
	\xymatrix{
		\Pb_f \otimes \Pb_f^- \ar[d]_-{\hat{\varphi} \,\otimes\, \hat{\varphi}^-} \ar[r]^-{\tau} & X \times \Aut{(D \otimes \K)^{\otimes 2}} \ar[d]^-{\lambda_{\rho}} \\
		\Pb_{f'} \otimes \Pb_{f'}^- \ar[r]^-{\tau'} & X \times \Aut{(D \otimes \K)^{\otimes 2}} \\
	}
\]
where $\lambda$ denotes the left multiplication. If the bundles are clear, we will sometimes abbreviate the tuple $(X,f,f^-,\tau)$ by $X$ and a morphism $(\varphi, \hat{\varphi}, \hat{\varphi}^-,\rho)$ by $\varphi$. Let $X,Y \in {\rm Obj}(\CW{D})$ and let $\varphi, \varphi' \colon X \to Y$ be two morphisms. A \emph{homotopy} between $\varphi$ and $\varphi'$ is a homotopy $H \colon X \times I \to Y$ between $\varphi$ and $\varphi'$ that preserves all the further structure in the sense that $H$ is covered by homotopies of principal bundle maps $\hat{H} \colon \Pb_f \times I \to \Pb_{f'}$ and $\hat{H}^- \colon \Pb_f^- \times I \to \Pb_{f'}^-$ and a corresponding homotopy of the changes of trivializations. 
\end{definition}

\begin{remark} \label{rem:heq}
Given a CW-complex $X$, an object $Y \in {\rm Obj}(\CW{D})$ and a map $\varphi \colon X \to Y$ all extra data can be pulled back to $X$ along $\varphi$. In particular, a subcomplex $A \subset X$ can be extended to an object $A \in {\rm Obj}(\CW{D})$.

If $\varphi \colon X \to Y$ as above is a homotopy equivalence with homotopy inverse $\psi \colon Y \to X$, then we can find maps of principal bundles $\hat{\varphi} \colon \varphi^*\Pb_f \to \Pb_f$ and $\hat{\psi} \colon \Pb_f \to \varphi^*\Pb_f$ covering $\varphi$ and $\psi$ respectively. This can be done in such a way that $\hat{\varphi} \circ \hat{\psi}$ and $\hat{\psi} \circ \hat{\varphi}$ are both homotopic to automorphisms of principal bundles covering the identity. A corresponding argument for the trivializations yields the following: Given an object $(Y,f,f^-,\tau) \in {\rm Obj}(\CW{D})$, a CW-complex $X$ and a homotopy equivalence $\varphi \colon X \to Y$, there are morphisms $X \to Y$ and $Y \to X$ in $\CW{D}$, such that both compositions are homotopic to automorphisms of $(X, f \circ \varphi, f^- \circ \varphi, \tau \circ \varphi)$ and $(Y,f,f^-, \tau)$, respectively. In particular, any homotopy invariant functor will map $(\varphi, \hat{\varphi}, \hat{\varphi}^-, \rho)$ to an isomorphism.
\end{remark}

\begin{theorem} \label{thm:det_by_coeff}
Let $h_{*}, h'_* \colon \CW{D} \to {\rm GrAb}$ be two functors to the category of graded abelian groups with the following properties:
\begin{enumerate}[(i)]
	\item \emph{homotopy invariance}: If $\varphi, \varphi' \colon X \to Y$ are morphisms such that $\varphi$ is homotopic to $\varphi'$ in the sense of Def.\ \ref{def:cat_CWD}, then $h_k(\varphi) = h_k(\varphi')$.
	\item \emph{Mayer-Vietoris sequence}: If $X \in {\rm Obj}(\CW{D})$ and $A \subset X$, $B \subset X$ are subcomplexes with $X = A \cup B$. Then there is a natural transformation $\partial \colon h_*(X) \to h_{*-1}(A \cap B)$ and the following sequence, in which the unlabeled arrows are induced by inclusions as in the ordinary Mayer-Vietoris sequence, is exact:
	\[
		\xymatrix@C=0.7cm{
			\ar[r] & h_k(A \cap B) \ar[r] & h_k(A) \oplus h_k(B) \ar[r] & h_k(X) \ar[r]^-{\partial} & h_{k-1}(A \cap B) \ar[r] &
		}
	\]
\end{enumerate}
Let $\eta \colon h_* \to h'_*$ be a natural transformation. If $\eta_{*} \colon h(*) \to h'(*)$ is an isomorphism of graded abelian groups for all objects $\ast \in {\rm Obj}(\CW{D})$ that have the one-point space as underlying CW-complex. Then $\eta_{X}$ is an isomorphism for all objects $X \in {\rm Obj}(\CW{D})$.
\end{theorem}

\begin{proof} \label{pf:det_by_coeff}
Let $X \in {\rm Obj}(\CW{D})$. The underlying CW-complex of $X$ is homotopy equivalent to a finite simplicial complex. After barycentric subdivision there exists a cover by subcomplexes $(U_i)_{i \in J}$, such that for every $k \in \N$ each intersection $U_{i_1} \cap \dots \cap U_{i_k}$ is either contractible or empty. (Such a cover will be called \emph{good} in the following.) Thus, we obtain a CW-complex $X'$, which has a good cover, and a homotopy equivalence $\varphi \colon X \to X'$ with homotopy inverse $\psi \colon X' \to X$. Using the pullback along $\psi$, we obtain $X' \in {\rm Obj}(\CW{D})$. The composition $\varphi \circ \psi$ is homotopic to an automorphism of $X' \in {\rm Obj}(\CW{D})$ and likewise for $\psi \circ \varphi$. Therefore $h(\varphi)$ is an isomorphism and it suffices to show the statement for objects, which allow good covers.

Suppose that we have shown that $\eta_{Y}$ is an isomorphism for those $Y \in {\rm Obj}(\CW{D})$, which has a good cover by $n$ subcomplexes. This holds for $n = 1$ by homotopy invariance and the assumption on $\eta$. Let $Y' \in {\rm Obj}(\CW{D})$ be an object which has a good cover by subcomplexes $V_1, \dots, V_{n+1}$, let $V = V_1 \cup \dots \cup V_n$ and note that $V$ and $V \cap V_{n+1}$ satisfy the induction hypothesis. The five lemma applied to the following comparison diagram of Mayer-Vietoris sequences 
\[
	\xymatrix@C=0.5cm{
		\ar[r] & h_k(V) \oplus h_k(V_{n+1}) \ar[r] \ar[d]^-{\cong} & h_k(V \cup V_{n+1}) \ar[r]^-{\partial} \ar[d]& h_{k-1}(V \cap V_{n+1}) \ar[r] \ar[d]_-{\cong} & \\
		\ar[r] & h'_k(V) \oplus h'_k(V_{n+1}) \ar[r] & h'_k(V \cup V_{n+1}) \ar[r]^-{\partial} & h'_{k-1}(V \cap V_{n+1}) \ar[r] & \\
	}
\]
proves the induction step and therefore the statement.
\end{proof}

The intermediate twisted homology theory alluded to in the introduction is defined as follows: Let $(X,f,f^-,\tau) \in {\rm Obj}(\CW{D})$ and consider the framed bordism group $\Omega^{\rm fr}_*((f^-)^*\Kbun{D})$. Given $(M,g \colon M \to (f^-)^*\Kbun{D}) \in \Omega^{\rm fr}_{n + 2k}((f^-)^*\Kbun{D})$, we obtain an element $(M \times S^2, g')$ (where the stable framing on $S^2$ is the one induced by $S^2 \subset \R^3$) in $\Omega^{\rm fr}_{n + 2(k+1)}((f^-)^*\Kbun{D})$, where $g'$ is given by
\[
	\xymatrix@C=1.5cm{
		g' \colon M \times S^2 \ar[r]^-{(g,b)} & (f^-)^*\Kbun{D} \times \hom_{\rm gr}(\Shat,\widehat{\K}) \ar[r]^-{m} & (f^-)^*\Kbun{D} 
	}
\]
and $b \colon S^2 \to \hom_{\rm gr}(\Shat, \widehat{\K})$ represents the Bott class.

\begin{definition} \label{def:intermediate}
Let $(X,f,f^-,\tau) \in {\rm Obj}(\CW{D})$, let $\Pb^- = \Pb_f^-$ and define	
\[
	k_n^{\Pb^-}(X) = \lim_{k} \Omega^{\rm fr}_{n + 2k}((f^-)^*\Kbun{D})\ .
\]
The direct limit is formed with respect to the maps $\Omega^{\rm fr}_{n+2k}((f^-)^*\KUbun{D}{2}) \to \Omega^{\rm fr}_{n+2(k+1)}((f^-)^*\KUbun{D}{2})$ described above. This is a functor $\CW{D} \to {\rm GrAb}$.
\end{definition}

\begin{lemma} \label{lem:satisfies_cond}
The functor $X \mapsto k_n^{\Pb^-}(X)$ satifies the conditions of Theorem \ref{thm:det_by_coeff}.
\end{lemma}

\begin{proof} \label{pf:satisfies_cond}
Let $(X,f,f^-,\tau),(Y,g,g^-,\tau') \in {\rm Obj}(\CW{D})$. A homotopy $H \colon X \times I \to Y$ is by definition covered by principal bundle maps $\Pb^-_f \times I \to \Pb^-_g$ inducing a homotopy $(f^-)^*\Kbun{D} \times I \to (g^-)^*\Kbun{D}$. Thus, homotopy invariance follows from the homotopy invariance of framed bordism.

Let $X \in {\rm Obj}(\CW{D})$ and let $A,B$ be subcomplexes, such that $X = A \cup B$. The double mapping cylinder $c(A,A \cap B,B) = (A \amalg (A \cap B) \times I \amalg B) /\!\!\sim$ is homotopy equivalent to $X$ via $\theta \colon c(A, A \cap B,B) \to X$. It can be covered by open sets $U$ and $V$, such that $(c(A, A \cap B, B), U, V)$ is an excisive triad, $U$ is homotopy equivalent to $A$, $V$ to $B$ and $U \cap V$ to $A \cap B$. All equivalences are induced by $\theta$. Pulling back the bundle $\Pb^-$ to $c(A,A \cap B,B)$ yields an excisive cover of $\theta^*(f^-)^*\Kbun{D}$ by $\left.\theta^*(f^-)^*\Kbun{D}\right|_U$ and $\left.\theta^*(f^-)^*\Kbun{D}\right|_V$. The observation in Rem.\ \ref{rem:heq} yields that $\Omega^{\rm fr}_{n+2k}(\left.\theta^*(f^-)^*\Kbun{D}\right|_U) \to \Omega^{\rm fr}_{n+2k}(\left.(f^-)^*\Kbun{D}\right|_A)$ is an isomorphism and likewise for the spaces over $A \cap B$, $B$ and $X$. Since direct limits preserve exact sequences, we therefore obtain the Mayer-Vietoris sequence for $k_*^{\Pb^-}$ from the one for framed bordism.
\end{proof}

\begin{definition} \label{def:Mfgr}
Let $KU^{D, {\rm mod,gr}}_n = \hom_{\rm gr}(\Shat, D \otimes \K \otimes \widehat{\K} \otimes (\C\ell_1 \otimes D \otimes \K)^{\otimes n})$. Analogous to $\KUmod{D}{\bullet}$ there also are action maps $\alpha_{\bullet,\bullet}$ such that $(KU^{D, {\rm mod,gr}}_{\bullet}, \alpha_{\bullet,\bullet})$ is a symmetric module spectrum over $KU^D_{\bullet}$ with the same properties as $\KUmod{D}{\bullet}$ by a similar proof as for Thm.\ \ref{thm:KU_module}. Let $\mathscr{K}U^{D, {\rm gr}}_n = E\Aut{D \otimes \K} \times_{\lambda} KU^{D, {\rm mod,gr}}_n$, where the left action acts on the $D \otimes \K$-factor of the tensor product. The \emph{graded Thom spectrum} of a map $f \colon X \to B\Aut{D \otimes \K}$ is defined by
\[
	Mf_n^{\rm gr} = r_!f^*\mathscr{K}U^{D, {\rm gr}}_n\ .
\]
It is a symmetric module spectrum over $KU^D_{\bullet}$ with respect to the action maps induced by $\alpha_{\bullet,\bullet}$. Let $\Pb = f^*E\Aut{D \otimes \K}$. We define $K_n^{\Pb,{\rm top, gr}}(X) = \pi_n(Mf^{\rm gr}_{\bullet})$.
\end{definition}

Let $e \in \widehat{\K}$ be a rank $1$-projection in the even part of the algebra. We have a $C^*$-algebra homomorphism $D \otimes \K \to D \otimes \K \otimes \widehat{\K}$ given by $a \mapsto a \otimes e$. Since this stabilization induces an isomorphism on $K$-groups, we obtain a $\pi_*$-equivalence of spectra $\KUmod{D}{\bullet} \to KU_{\bullet}^{D,{\rm mod, gr}}$ and a corresponding map of symmetric $KU^D_{\bullet}$-module spectra $Mf_{\bullet} \to Mf^{\rm gr}_{\bullet}$.

\begin{lemma} \label{thm:Mf_is_Mfgr}
The map $Mf_{\bullet} \to Mf^{\rm gr}_{\bullet}$ induces a natural isomorphism of the associated homology theories 
$
	K_{\bullet}^{\Pb,{\rm top}} \to K_{\bullet}^{\Pb, {\rm top, gr}} .
$
In particular, it is a $\pi_*$-equivalence and therefore a stable one.
\end{lemma}

\begin{proof} \label{pf:Mf_is_Mfgr}
By Thm.\ \ref{thm:Mayer-Vietoris} the functors $K_{\bullet}^{\Pb,{\rm top}}$ and $K_{\bullet}^{\Pb,{\rm top, gr}}$ on $\CW{D}$ satisfy the conditions of Thm.\ \ref{thm:det_by_coeff}. Therefore it suffices to check that $K_n^{\Pb, {\rm top}}(pt) \to K_n^{\Pb, {\rm top, gr}}(pt)$ is an isomorphism. But over a point we have $Mf_{\bullet} = \KUmod{D}{\bullet}$ and $Mf^{\rm gr}_{\bullet} = KU^{D, {\rm mod, gr}}_{\bullet}$ and the statement follows from the observation above.
\end{proof}

Multiplication by the Bott element $[b] \in K_2(pt)$ induces an isomorphism of twisted $K$-homology groups $KK_{-n}(C(X,\mathcal{A}),D) \to KK_{-(n+2)}(C(X,\mathcal{A}),D)$. The element $[b]$ can be represented by a pointed map $S^2 \to \hom_{\rm gr}(\Shat,\widehat{\K})$. Using the same construction as for $\Kbun{D}$, there is an action map
\[
	Mf^{\rm gr}_n \wedge \hom_{\rm gr}(\Shat, \widehat{\K}) \to Mf^{\rm gr}_n\ .
\]
Just as for $\Kbun{D}$ this yields a Bott map $Mf^{\rm gr}_n \to \Omega^2Mf^{\rm gr}_n$, which induces
\[
	K_n^{\Pb, {\rm top, gr}}(X) = \pi_n(Mf^{\rm gr}_{\bullet}) \to \pi_{n}(\Omega^2Mf_{\bullet}^{\rm gr}) \cong \pi_{n+2}(Mf_{\bullet}^{\rm gr}) = K_{n+2}^{\Pb, {\rm top, gr}}(X)\ . 
\]

\begin{definition}
We define 
\begin{align*}
	K_n^{\Pb, {\rm an,lim}}(X) & = \lim_k KK_{-(n + 2k)}(C(X, \mathcal{A}), D) \\
	K_n^{\Pb^-, {\rm top,lim}}(X) & = \lim_k K^{\Pb^-,{\rm top, gr}}_{n + 2k}(X)
\end{align*}
where the direct limit in both cases is taken over the Bott homomorphisms described above. \end{definition}

These are both functors $\CW{D} \to {\rm GrAb}$. It is immediate that $K_n^{\Pb, {\rm an}}(X) \to K_n^{\Pb, {\rm an,lim}}(X)$ is an isomorphism. The same holds true for $K_n^{\Pb^-, {\rm top, gr}}(X) \to K_n^{\Pb^-, {\rm top,lim}}(X)$ by Thm.\ \ref{thm:det_by_coeff} and the fact that the Bott homomorphism induces an isomorphism $K_n^{\Pb^-, {\rm top, gr}}(pt) \to K_{n+2}^{\Pb^-, {\rm top, gr}}(pt)$ on coefficients.

\begin{theorem} \label{thm:nat_iso_1}
Let $D$ be a strongly self-absorbing $C^*$-algebra, which satisfies the UCT. The analytic index map induces a natural isomorphism of functors
\[
	{\rm ind}_a \colon k_n^{\Pb^-} \to K_n^{\Pb,{\rm an,lim}}\ .
\]
\end{theorem}

\begin{proof}
The Poincar\'{e} duality homomorphism (\ref{eqn:PD}) required the choice of isomorphisms 
\[
	C(X, \mathcal{A} \otimes \mathcal{A}^-) \cong C(X, (D \otimes \K)^{\otimes 2}) \cong C(X, D \otimes \K)
\]
Since $\Pb$, its inverse $\Pb^-$ and the trivialization $\tau \colon \Pb \otimes \Pb^- \to X \times \Aut{(D \otimes \K)^{\otimes 2}}$ are part of the objects in $\CW{D}$, the first isomorphism is canonical. For the second identification there is a canonical path connected space of choices given by pairs of isomorphisms $D \otimes D \cong D$ and $\K \otimes \K \cong \K$ \cite[Thm.\ 2.3]{paper:DadarlatP1}. Moreover, a stable framing of a manifold $M$ determines a spin$^c$ structure on it. By Thm.\ \ref{thm:bordism_invariance} we therefore obtain a natural transformation 
\[
	\Omega^{\rm fr}_{n}((f^-)^*\Kbun{D}) \to K^{\Pb,{\rm an}}_n(X) \quad ; \quad [M,g] \mapsto {\rm ind}_a(M,g)\ .
\]
To check that this construction is compatible with direct limits over the Bott isomorphisms note that $[E_{m \circ (g \times \eta_2)}] = [E_g] \boxtimes [B] \in K^0_{\Pb}(M \times S^2)$, where $[B]\in K_0(S^2)$ is the Bott bundle (i.e.\ the Hopf bundle minus the trivial line bundle) and the tensor product is induced by the exterior tensor product. Let $q \colon S^2 \to \{pt\}$ denote the map to the point. The product with $q_*(\iota_{S^2}([B]) \otimes [S^2]) = [b] \in K_2(pt)$ is the Bott isomorphism. We have
\begin{align*}
{\rm ind}_a(M \times S^2, m \circ (g \times \eta_2)) & = (\bar{g}_* \circ {\rm pr}_{M*})PD_0([E_{m \circ (g \times \eta_2)}]) \\
& = (\bar{g}_* \circ {\rm pr}_{M*})\left( \left(\iota_{M}([E_g]) \otimes [M] \right) \boxtimes \left(\iota_{S^2}([B]) \otimes [S^2] \right) \right)\\
& = \bar{g}_* PD_0([E_{g}]) \cdot [b] = {\rm ind}_a(M,g) \cdot [b]
\end{align*}
Therefore we obtain a natural transformation $k_n^{\Pb^-} \to K_n^{\Pb, {\rm an,lim}}$. Since the right hand side is homotopy invariant and has Mayer-Vietoris sequences by \cite[Thm.\ 21.5.1]{book:Blackadar}, the conditions of Thm.~\ref{thm:det_by_coeff} are satisfied and it remains to be checked that $k_n^{\Pb^-}(pt) \to K_n^{\Pb,{\rm an,lim}}(pt)$ is an isomorphism, where $\Pb = \Aut{D \otimes \K}$ is the trivial bundle over $pt$. By the Pontrjagin-Thom construction the $n$-th framed cobordism group $\Omega^{\rm fr}_n(W)$ of an unbased space $W$ agrees with the stable homotopy group $\pi^s_n(W_+)$, where $W_+$ is $W$ with a disjoint basepoint. Let $K^D = \hom_{\rm gr}(\Shat, D \otimes \K \otimes \widehat{\K})$. Using the same argument as in \cite[Sec.\ 9.1]{paper:BaumHigsonSchick} (see the remark below) it follows that 
\[
	k^{\Pb}_n(pt) = \lim_k \Omega^{\rm fr}_{n+2k}(K^D) \cong \lim_k \pi^s_{n+2k}((K^D)_+) \cong \pi_n(K^D)\ .
\]
The unit map induces an isomorphism $KK_i(D,D) \cong K_i(D)$ by \cite[Thm.\ 3.4]{paper:DadarlatKK1}. The homomorphism $\pi_0(K^D) \to k^{\Pb}_0(pt) \to K_0^{\Pb,{\rm an,lim}}(pt) \cong KK(D,D) \cong K_0(D)$ sends a point in $K^D$ to the $K$-theory class it represents, possibly altered by an automorphism of $D \otimes \K$ that arises from the trivialization of $\Pb \otimes \Pb^- \to pt$. This is an isomorphism. Since we assumed that $D$ satisfies the UCT, we have $KK_1(D,D) \cong K_1(D) \cong 0$ by \cite[Prop.\ 5.1]{paper:TomsWinter}, which implies that $0 = k^{\Pb}_1(pt) \to K^{\Pb, {\rm an,lim}}_1(pt) = 0$ is an isomorphism. 
\end{proof}

\begin{remark} \label{rem:basept_missing}
The basepoint is missing in the argument given in \cite[Sec.\ 9.1]{paper:BaumHigsonSchick}, i.e.\ it should be $\pi_{n+2k}(S^{2m} \wedge \K_+)$ in the square diagram on page 21. The calculation of the homotopy groups is nevertheless correct. This follows from the fact that the map $b \colon S^2 \wedge \K_+ \to \K_+$ induced by Bott periodicity factors through $\K$ (note that $\K$ already has a basepoint). The induced homomorphism $\pi_{n+2k}(S^{2m} \wedge \K_+) \to \pi_{n+2k+2}(S^{2m} \wedge \K)$ yields an isomorphism of the direct limits. We obtain  $\lim_k \pi^s_{n+2k}((K^D)_+) \cong \pi_n(K^D)$ in the same way.
\end{remark}

Observe that there are canonical maps $\mu \colon ((f^-)^*\Kbun{D})_+ \wedge KU^D_n \to r_!(f^-)^*\mathscr{K}U^{D,{\rm gr}}_n = Mf_n^{-,{\rm gr}}$, which yield a corresponding map of symmetric spectra. Let $\mathbb{S}_{\bullet}$ be the sphere spectrum and consider  $\mathbb{S}_{\bullet} \wedge ((f^-)^*\Kbun{D})_+ \to Mf^{-,{\rm gr}}_{\bullet}$ given by
\[
	\xymatrix@C=1.5cm{
		S^n \wedge ((f^-)^*\Kbun{D})_+ \ar[r]^-{\eta_n \wedge \id{}} & KU^D_n \wedge ((f^-)^*\Kbun{D})_+ \ar[r]^-{\mu \circ \tau} & Mf^{-,{\rm gr}}_n
	}
\]
where $\tau$ just switches the factors in the wedge product. By the Pontryagin-Thom construction we have $\Omega^{\rm fr}_n((f^-)^*\Kbun{D}) \cong \pi_n^s(((f^-)^*\Kbun{D})_+) \cong \pi_n(\mathbb{S}_{\bullet} \wedge ((f^-)^*\Kbun{D})_+)$
and the above map induces a transformation $\Omega^{\rm fr}_n((f^-)^*\Kbun{D}) \to \pi_n(Mf^{-,{\rm gr}}_{\bullet}) = K^{\Pb^-, {\rm top, gr}}_n(X)$. The limits on both sides were designed in such a way that this extends to a natural transformation
\[
	k^{\Pb^-}_{n} \to K^{\Pb^-, {\rm top, lim}}_{n}\ .
\]

\begin{theorem} \label{thm:nat_iso_2}
The above transformation is a natural isomorphism $k^{\Pb^-}_{n} \to K^{\Pb^-, {\rm top, lim}}_{n}$. If $D$ is a strongly self-absorbing $C^*$-algebra that satisfies the UCT, we obtain a natural isomorphism of functors $\CW{D} \to {\rm GrAb}$ of the form
\[
	K^{\Pb^-, {\rm top}}_{\bullet} \cong K^{\Pb, {\rm an}}_{\bullet} \ .
\]
\end{theorem}

\begin{proof} \label{pf:nat_iso_2}
We have seen in Lem.\ \ref{lem:satisfies_cond} and Thm.\ \ref{thm:Mayer-Vietoris} that the functors $k_{\bullet}^{\Pb^-}$ and $K_{\bullet}^{\Pb^-, {\rm top, lim}}$ satisfy the conditions of Thm.\ \ref{thm:det_by_coeff}. To prove the first statement, it therefore suffices to check that $k^{\Pb^-}_n(pt) \to K_n^{\Pb^-, {\rm top,lim}}(pt)$ is an isomorphism. In the case $X = pt$ we have $Mf^{-,{\rm gr}}_{\bullet} = KU^{D, {\rm mod,gr}}_{\bullet}$, in particular $Mf^{-,{\rm gr}}_0 = K^D$ in the notation of the proof of Thm.\ \ref{thm:nat_iso_1}. As already stated there we have
\[
	k^{\Pb^-}_n(pt) \cong \lim_k \pi_{n+2k}^s((K^D)_+) \cong \pi_n(K^D)
\]
and the following square commutes
\[
	\xymatrix{
		\pi_n(K^D) \ar[d]_-{\cong} \ar[r]^-{\cong} & \pi_n(Mf^{-,{\rm gr}}_{0}) \ar[d]^-{\cong} \\
		k^{\Pb^-}_n(pt) \cong \lim_k \pi_{n+2k}^s((K^D)_+) \ar[r] & \lim_k \pi_{n+2k}(Mf^{-,{\rm gr}}_{\bullet}) \cong K^{\Pb^-,{\rm top,lim}}(pt)
	}
\]
proving the first claim. For the second statement we combine all natural isomorphisms
\[
	\xymatrix{
		K^{\Pb^-, {\rm top}}_{\bullet} \ar[r]^-{\cong} & K^{\Pb^-, {\rm top, gr}}_{\bullet} \ar[r]^-{\cong} & K^{\Pb^-, {\rm top, lim}}_{\bullet} & \ar[l]_-{\cong} k_{\bullet}^{\Pb^-} \ar[r]^-{\cong} & K^{\Pb, {\rm an, lim}}_{\bullet} & \ar[l]_-{\cong} K^{\Pb, {\rm an}}_{\bullet}  
	}
\]
proven above.
\end{proof}

\begin{remark} \label{rem:UCT}
It is conjectured that every separable nuclear $C^*$-algebra satisfies the universal coefficient theorem (UCT) in $KK$-theory needed above and all the known examples of strongly self-absorbing $C^*$-algebras do. 
\end{remark}

\section{Higher twisted $K$-Theory of suspensions}
Let $D$ be a strongly self-absorbing $C^*$-algebra, which satisfies the universal coeffient theorem.  In this section we will construct a noncommutative model of all twists of $K$-Theory for suspensions and compute the corresponding twisted $K$-groups, in particular those of the odd-dimensional spheres. Observe that $gl_1(KU^D)^1(S^{2k}) = [S^{2k}, BGL_1(KU^D)] = 0$, therefore the twisted $K$-groups of even-dimensional spheres agree with the untwisted ones. 

\subsection{Morita equivalences and linking algebras}
Given two $C^*$-algebras $A$ and $B$, an $A$-$B$-\emph{(Morita) equivalence} is a full right Hilbert $B$-module $X$ together with an isomorphism $A \to \cpt{B}{X}$ onto the compact $B$-linear operators on $X$. Equivalently, $X$ is a Banach space that is a full left Hilbert $A$-module as well as a full right Hilbert $B$-module, such that the inner products satisfy $\lscal{A}{\xi}{\eta}\zeta = \xi\rscal{B}{\eta}{\zeta}$. If $X$ is an $A$-$B$-equivalence, then the adjoint module $X^*$ is a $B$-$A$-equivalence. This is the inverse equivalence in the sense that the inner products induce canonical bimodule isomorphisms $X \otimes_B X^* \cong A$ and $X^* \otimes_A X \cong B$. We refer the reader to \cite{paper:RieffelMorita, paper:BrownGreenRieffel} for a detailed exposition. We can associate to an $A$-$B$-equivalence $X$ the \emph{linking algebra} $L(X)$, which contains $A$ and $B$ as complementary full corners \cite[Thm.\ 1.1]{paper:BrownGreenRieffel}. This is given by
\[
	L(X) = 
	\begin{pmatrix} 
		A & X \\
		X^* & B
	\end{pmatrix}
\]
with the multiplication and $*$-operation:
\[
	\begin{pmatrix}
		a_1 & \xi_1 \\
		\eta^*_1 & b_1
	\end{pmatrix}
	\cdot
	\begin{pmatrix}
		a_2 & \xi_2 \\
		\eta^*_2 & b_2
	\end{pmatrix} = 
	\begin{pmatrix}
		a_1a_2 + \lscal{A}{\xi_1}{\eta_2} & a_1\xi_2 + \xi_1b_2 \\
		\eta^*_1a_2 + b_1\eta_2 &  \rscal{B}{\eta_1}{\xi_2} + b_1b_2
	\end{pmatrix}
	\quad , \quad
	\begin{pmatrix}
		a & \xi \\
		\eta^* & b
	\end{pmatrix}^*
	= 
	\begin{pmatrix}
		a^* & \eta \\
		\xi^* & b^*
	\end{pmatrix}\ .
\]
If $A$ and $B$ are unital, then $X$ is finitely generated and projective as a right Hilbert $B$- and a left Hilbert $A$-module. In this case we have an induced isomorphism $X_* \colon K_0(A) \to K_0(B)$, which sends a finitely generated projective right Hilbert $A$-module $E$ to $E \otimes_A X$. 

Let $H_B = \ell^2(\N) \otimes B$. If we choose an isomorphism of right Hilbert $B$-modules $X \cong pH_B$ for a projection $p \in B \otimes \K$, then the left action of $A$ on $X$ identifies $A$ with the corner $p(B \otimes \K)p \subset B \otimes \K$. In particular, $X$ induces an injective $*$-homomorphism $\iota_X \colon A \otimes \K \to B \otimes \K$. Let $q \in A \otimes \K$ be a projection with $E \cong qH_A$, then $X_*(E)$ corresponds to the projection $\iota_X(q)$. 

Let $p_0 = \left(\begin{smallmatrix} 1 & 0 \\ 0 & 0  \end{smallmatrix}\right) \in L(X)$ and $p_1 = \left(\begin{smallmatrix} 0 & 0 \\ 0 & 1 \end{smallmatrix} \right) \in L(X)$. $Y_0 = p_0L(X)$ is a Morita equivalence between $A = p_0L(X)p_0$ and $L(X)$. Similarly, $Y_1= p_1L(X)$ is a $B$-$L(X)$-equivalence. The following diagram of isomorphisms commutes:
\[
	\xymatrix{
	 K_0(A) \ar[dr]_-{Y_{0*}} \ar[rr]^-{X_*} & & K_0(B) \ar[dl]^-{Y_{1*}}\\
	 & K_0(L(X))
	}
\]
As above, $Y_{0*}$ agrees with the map that sends a projection $q \in A \otimes \K$ to $\left(\begin{smallmatrix} q & 0 \\ 0 & 0 \end{smallmatrix}\right) \in L(X) \otimes \K$. Apart from $Y_{0*}$ we have another natural homomorphism $K_0(A) \to K_0(L(X))$ induced by a stable isomorphism $\varphi_0 \colon A \otimes \K \to L(X) \otimes \K$ constructed as follows: Consider $p_0 \otimes 1 \in M(L(X) \otimes \K)$. By \cite[Lem.\ 2.5]{paper:BrownStableIso} there exists a partial isometry $v \in M(L(X) \otimes \K)$ with $v^*v = 1$ and $vv^* = p_0 \otimes 1$, which induces $\varphi_0$ via $\varphi_0(a) = v^*\left(\begin{smallmatrix} a & 0 \\ 0 & 0 \end{smallmatrix}\right)v$. Analogously, we define the isomorphism $\varphi_{1} \colon B \otimes \K \to L(X) \otimes \K$.

\begin{lemma} \label{lem:phi_is_Y}
The isomorphism $\varphi_{0*} \colon K_0(A) \to K_0(L(X))$ does not depend on the choice of $v$, $v^*$ with the above properties. Moreover, we have $Y_{i*} = \varphi_{i*}$ for $i \in \{0,1\}$ and $(\varphi_{1}^{-1} \circ \varphi_{0})_* = X_*$.
\end{lemma}

\begin{proof} \label{pf:phi_is_Y}
If $w \in M(L(X) \otimes \K)$ is another partial isometry with $w^*w = 1$ and $ww^* = p_0 \otimes 1$, then $u = v^*w$ is a unitary with $vu = w$. Since $U(M(L(X) \otimes \K))$ is path-connected \cite{paper:CuntzHigson}, the two homomorphisms $K_0(A) \to K_0(L(X))$ corresponding to $v$ and $w$ agree. 

Let $\iota \colon A \otimes \K \to L(X) \otimes \K$ be given by $\iota(a) = \left(\begin{smallmatrix} a & 0 \\ 0 & 0 \end{smallmatrix} \right)$. By the above remarks we have $\iota_* = Y_{0*}$. Moreover, $\iota \circ \varphi_0^{-1}(x) = vxv^*$. We will construct a strictly continuous path $v_t \in M(L(X) \otimes \K)$ for $t \in [0,1]$ with $v_0 = v$ and $v_1 = 1$. This will prove the second claim. Let $H_L = \ell^2(\N) \otimes L(X)$. We can identify $M(L(X) \otimes \K)$ with $\bdd{L(X)}{H_L}$, i.e.\ the bounded adjointable operators on $H_L$. There are two notions of strict topology on both algebras, which agree on the unit ball \cite[Prop.\ 8.1]{book:Lance}.

Choose an isomorphism $H_L \oplus H_L \to H_L$ and denote by $r_0$, $r_1 \colon H_L \to H_L$ the precomposition of this isomorphism with the natural maps $H_L \to H_L \oplus H_L$. In particular, we have $r_i^*r_j = \delta_{ij}1$. Using the same idea as in the proof of \cite[Thm.\ I.3.2.10]{book:BlackadarOpAlg} we obtain a strictly continuous path of partial isometries $w_t \colon H_L \to H_L$ with $w_t^* w_t = 1$, $w_1 = r_0$ and $w_0 = 1$. Denote the corresponding path for $r_1$ by $w_t''$. Now consider $w_t' = \sqrt{1-t}\,r_0v + \sqrt{t}\,r_1$, which is a strictly continuous path of partial isometries with $(w_t')^*w_t' = 1$, $w_0' = r_0v$ and $w_1' = r_1$. $v_t$ is now given by concatenating the paths $w_t\,v$, $w_t'$ and $w_{1-t}''$.
\end{proof}

\subsection{Construction of higher twists}
Let $X = \Sigma Y$ be the suspension of a finite-dimensional compact metrizable space $Y$. We will assume that $Y$ is cofibrantly pointed so that the map from the unreduced to the reduced suspension is a homotopy equivalence. Choose contractible closed subsets $W_0, W_1 \subset X$, contractible open subsets $U_i \subset W_i$ and contractible closed subsets $V_i \subset U_i$, such that each of the pairs $W_0, W_1$, $U_0, U_1$ and $V_0, V_1$ covers $X$ and such that $W_0 \cap W_1 \simeq Y$, $U_0 \cap U_1 \simeq Y$ and $V_0 \cap V_1 \simeq Y$. The twists of the $K$-Theory of $X$ are given by 
\[
	gl_1(KU^D)^1(\Sigma Y) = [Y, \Omega B\Aut{D \otimes \K}] = [Y, \Aut{D \otimes \K}] \cong K_0(C(Y) \otimes D)^{\times}_+\ .
\]
In the following we will make this isomorphism precise by constructing a continuous $C(X)$-algebra $A$ with fibers stably isomorphic to $D$, which satisfies the generalized Fell condition \cite[Def.\ 4.1]{paper:DadarlatP1}. By \cite[Thm.\ 4.2]{paper:DadarlatP1}, $A \otimes \K$ is then isomorphic to the section algebra of a locally trivial continuous bundle with fiber $D \otimes \K$. 

Let $B = C(X) \otimes D \otimes \K$, $B_0 = C_0(U_1 \cap U_2) \otimes D \otimes \K$ and let $p \in B(W_0 \cap W_1)$ be a full projection representing an invertible element $[p] \in K_0(B(W_0 \cap W_1))_+^{\times}$. By \cite[Thm.\ 0.1]{paper:DadarlatWinterTriv} and \cite[Lem.\ 2.14]{paper:DadarlatP1} the continuous $C(W_0 \cap W_1)$-algebra $pB(W_0 \cap W_1)p$ is isomorphic to $C(W_0 \cap W_1) \otimes D$. Let $q_0 = 1_{C(X)} \otimes 1_D \otimes e \in B$ and let $\varphi \colon C(W_0 \cap W_1) \otimes D \to pB(W_0 \cap W_1)p$ be an isomorphism, which turns $H' = pB(W_0 \cap W_1)q_0$ into a self-Morita equivalence of $C(W_0 \cap W_1) \otimes D$. Let $H'_0 = pB_0q_0 \subset H'$. Since $\varphi$ restricts to an isomorphism $C_0(U_0 \cap U_1) \otimes D \to pB_0p$, $H_0'$ is a self-Morita equivalence of $C_0(U_0 \cap U_1) \otimes D$. We can extend $H_0'$ to a (non-full) Hilbert $C(X) \otimes D$-bimodule $H$ by defining the multiplication with functions in $C(X) \otimes D$ via restriction. The inner product takes values in $C_0(U_0 \cap U_1) \otimes D$, which sits inside $C(X) \otimes D$ via extension by $0$. Note that the fiber $H(x) = H / C_0(X\setminus \{x\})H$ is zero if $x \notin U_0 \cap U_1$. Let $H^*$ be the analogous extension of $(H_0')^*$. Let $A_0 = C_0(U_0) \otimes D$ and extend it to a continuous $C(X)$-algebra in the same way, likewise define $A_1 = C_0(U_1) \otimes D$.  Now consider the linking algebra of $H$ given by 
\[
	A = 
	\begin{pmatrix} 
		A_0 & H \\
		H^* & A_1
	\end{pmatrix}
\]

\begin{lemma} \label{lem:A_cont_C(X)}
The $C^*$-algebra $A$ is a continuous $C(X)$-algebra. All fibers of $A \otimes \K$ are isomorphic to $D \otimes \K$ and $A \otimes \K$ satisfies the generalized Fell condition. Therefore there is a locally trivial bundle of $C^*$-algebras $\mathcal{A} \to X$ with fiber $D \otimes \K$, such that $C(X,\mathcal{A}) \cong A$ as $C(X)$-algebras.
\end{lemma}

\begin{proof} \label{pf:A_cont_C(X)}
By construction $A$ is a continuous $C(X)$-algebra. The statement about the fibers is clear for all $x \in X \setminus U_i$. For $x \in U_0 \cap U_1$ the fiber is given by the stabilization of the linking algebra $L(H(x)) = \left( \begin{smallmatrix} D & H(x) \\ H(x)^* & D \end{smallmatrix} \right)$. But this is also isomorphic to $D \otimes \K$. Consider the closed cover of $X$ by $V_0, V_1$ and note that $A(V_0) \otimes \K$ contains the projection $p_0 = \left(\begin{smallmatrix} 1 & 0 \\ 0 & 0 \end{smallmatrix} \right) \otimes e$, likewise we have $p_1 = \left(\begin{smallmatrix} 0 & 0 \\ 0 & 1 \end{smallmatrix} \right) \otimes e \in A(V_1) \otimes \K$. Both projections represent invertible elements in $K_0(A(x))$. Thus, $A$ satisfies the generalized Fell condition. The existence of $\mathcal{A}$ now follows from \cite[Thm.\ 4.2]{paper:DadarlatP1}.
\end{proof}

\begin{lemma} \label{lem:class_is_p}
Let $\mathcal{A} \to X$ be the locally trivial bundle from the last lemma. The isomorphism $\kappa \colon gl_1(KU^D)^1(X) = [X, B\Aut{D \otimes \K}] \to K_0(C(Y) \otimes D)^{\times}_+$ maps $[\mathcal{A}]$ to $[p]$.
\end{lemma}

\begin{proof} \label{pf:class_is_p}
Let $\phi \colon V_0 \cap V_1 \to \Aut{D \otimes \K}$ be the transition function of $\mathcal{A} \to X$ corresponding to a choice of trivialization of $\mathcal{A}$ over $V_0$ and $V_1$. Since $X$ is a suspension, $\phi$ completely determines the isomorphism class of $\mathcal{A}$. It induces a $*$-homomorphism $\hat{\phi} \colon D \otimes \K \to C(V_0 \cap V_1) \otimes D \otimes \K$. By definition of $\kappa$ we have $\kappa([\mathcal{A}]) = [\hat{\phi}(1 \otimes e)] \in K_0(C(V_0 \cap V_1) \otimes D)_+^{\times}$ and the class does not depend on the trivializations. Choose an isomorphism $C(X, \mathcal{A}) \to A \otimes \K$. In particular, we obtain trivializations $C(V_i, \left.\mathcal{A}\right|_{V_i}) \to A(V_i) \otimes \K \to C(V_i) \otimes D \otimes \K $ using the projections $p_i \in A(V_i) \otimes \K$ to construct partial isometries as described in the paragraph preceding Lemma~2.3. Using the fact that $A(V_0 \cap V_1) \otimes \K = L(H(V_0 \cap V_1)) \otimes \K$ if follows from this lemma that the homomorphism $\hat{\phi}$ corresponding to those trivializations maps $1 \otimes e$ to the right Hilbert $C(V_0 \cap V_1) \otimes D$-module $H(V_0 \cap V_1) = p\,H_{C(V_0 \cap V_1) \otimes D}$, which is precisely the one represented by the projection $p$.
\end{proof}

\begin{remark} \label{rem:Fell_bundle}
Let $G^{(1)} = \amalg_{i,j}\,U_{i} \cap U_j$, $G^{(0)} = \amalg_i\,U_i$ with $i,j \in \{0,1\}$. These form the morphism and object space of the \v{C}ech groupoid associated to the open cover $U_0 \cup U_1 = X$. Let $F \to G^{(1)}$ be the Fell bundle given by $F_{00} = U_0 \times D$ over $U_0 \cap U_0$, $F_{11} = U_1 \times D$ over $U_1 \cap U_1$, 
\[
	F_{01} = \{ (x, \xi) \in U_0 \cap U_1 \times H_D\ |\ p(x)\xi = \xi \}
\]
over $U_0 \cap U_1$ and $F_{10} = F_{01}^*$ (applied fiberwise) over $U_1 \cap U_0$. The multiplication of this Fell bundle is either given by the multiplication in the algebra or by left (respectively right) multiplication on the bimodule bundle. The above algebra $A$ corresponds to the $C^*$-algebra $C^*(F)$ associated to this Fell bundle. This point of view allows us to generalize the construction to spaces which are not suspensions. 
\end{remark}

Observe that $K_0(C(Y) \otimes D)$ is a right module over $K_0(D)$. Given $[p] \in K_0(C(Y) \otimes D)$ and $[q] \in K_0(D)$, we will denote the action by $[p] \cdot [q]$. The reduced $K$-Theory $\widetilde{K}_0(C(Y) \otimes D)$ of a compact space $Y$ with coefficients in $D$ is defined to be the cokernel of the map $K_0(D) \to K_0(C(Y) \otimes D)$ induced by the unital homomorphism $d \mapsto 1_{C(Y)} \otimes d$.

\begin{theorem} \label{thm:twistedK_of_suspension}
Let $D$ be a strongly self-absorbing $C^*$-algebra. Let $\mathcal{A} \to X$ be the bundle of $C^*$-algebras associated to a projection $p \in C(Y) \otimes D \otimes \K$ by the construction described above. Let $G$ be the additive subgroup given by $G = \{ x \in K_0(D)\ |\ [p] \cdot x = [1] \cdot y \in K_0(C(Y) \otimes D) \text{ for some } y \in K_0(D) \}$. Then we have a short exact sequence
\[
	0 \to K_1(C(Y) \otimes D) \to K^0_{\Pb}(X) \to G \to 0\ .
\]
where the surjection is induced by evaluation at a point $x_0 \in X$. Moreover, $K^1_{\Pb}(X) \cong \widetilde{K}_0(C(Y) \otimes D) / ([p])$.
\end{theorem}

\begin{proof} \label{pf:twistedK_of_suspension}
Note that $C(V_0 \cap V_1, \left.\mathcal{A}\right|_{V_0 \cap V_1}) \cong A(V_0 \cap V_1) \otimes \K \cong C(V_0 \cap V_1) \otimes D \otimes \K$. Since $V_0$ and $V_1$ are contractible, $K_1(D) = 0$ and $K_0(C(V_0 \cap V_1)\otimes D) \cong K_0(C(Y) \otimes D)$, the Mayer-Vietoris six-term exact sequence of $X$ boils down to 
\[
	\xymatrix{
		K_0(C(X,\mathcal{A})) \ar[r] & K_0(D) \oplus K_0(D) \ar[r]^-{\theta} & K_0(C(Y) \otimes D) \ar[d] \\
		K_1(C(Y) \otimes D) \ar[u] & 0 \ar[l] & K_1(C(X,\mathcal{A})) \ar[l]
	}
\]
We have $K_0(C(V_0 \cap V_1) \otimes D) \cong K_0(C(Y) \otimes D)$ and we can identify $C(V_0 \cap V_1,\left.\mathcal{A}\right|_{V_0 \cap V_1})$ with $A(V_0 \cap V_1) \otimes \K = L(H(V_0 \cap V_1)) \otimes \K$. With the notation from the paragraph preceding Lemma~\ref{lem:phi_is_Y} we obtain as a consequence of that lemma that 
\begin{equation} \label{eqn:theta}
	\theta(x,y) = \varphi_{1*}^{-1}(Y_{0*}(x) - Y_{1*}(y)) =  H_*(x) - [1] \cdot y = [p] \cdot x - [1] \cdot y\ .
\end{equation}
Thus, $\ker(\theta) \cong G$, which implies the existence of the claimed short exact sequence for $K_0(C(X,\mathcal{A}))$. The statement about $K_1(C(X,\mathcal{A}))$ is clear from (\ref{eqn:theta}).
\end{proof}

Let $[H] \in K^0(S^{2k})$ be the class represented by the canonical line bundle $H$ over $\C P^1$. We have $K^1(S^{2k}) = 0$ and $K^0(S^{2k}) \cong \Z[t]/(t^2)$, where $t = [H]-[1]$. Since $D$ satisfies the UCT, we have $K_0(C(S^{2k}) \otimes D) \cong K_0(D)[t]/(t^2)$. Invertible elements in this ring are of the form $a + bt$ with $a,b \in K_0(D)$ and $a$ invertible.

\begin{corollary} \label{thm:twistedK_of_spheres}
Let $D$ be a strongly self-absorbing $C^*$-algebra that satisfies the UCT. Let $\mathcal{A} \to S^{2k+1}$ be a locally trivial bundle with fiber $D \otimes \K$ associated to a projection $p \in C(S^{2k} \otimes D \otimes \K)$ as described above. Define $a,b \in K_0(D)$ by $[p] = a + bt \in K_0(C(S^{2k}) \otimes D)$ using the isomorphism from the last paragraph and suppose that $b \neq 0$. Then
\begin{align*}
K^0_{\Pb}(S^{2k+1}) & = 0 \ ,\\
K^1_{\Pb}(S^{2k+1}) & \cong K_0(D) / (b)\ .
\end{align*}
\end{corollary}

\begin{proof} \label{pf:twistedK_of_spheres}
If $b \neq 0$, the group $G$ from Theorem~\ref{thm:twistedK_of_suspension} is trivial. Moreover, $K_1(C(S^{2k}) \otimes D) = 0$. Thus, $K^0_{\Pb}(S^{2k+1}) = 0$. The reduced $K$-group $\widetilde{K}_0(C(S^{2k})\otimes D)$ is isomorphic to $K_0(D)$ and the quotient map $K_0(C(S^{2k})\otimes D) \to \widetilde{K}_0(C(S^{2k})\otimes D) \cong K_0(D)$ can be identified with $(x + yt) \mapsto y$. This implies the statement about $K^1_{\Pb}(S^{2k+1})$.
\end{proof}

\bibliographystyle{plain}
\bibliography{TwistedK}

\begin{thebibliography}{10}

\bibitem{paper:Ando}
Matthew Ando, Andrew~J. Blumberg, and David Gepner.
\newblock Twists of {$K$}-theory and {TMF}.
\newblock In {\em Superstrings, geometry, topology, and {$C^\ast$}-algebras},
  volume~81 of {\em Proc. Sympos. Pure Math.}, pages 27--63. Amer. Math. Soc.,
  Providence, RI, 2010.

\bibitem{paper:ABGH}
Matthew Ando, Andrew~J. Blumberg, David Gepner, Michael~J. Hopkins, and Charles
  Rezk.
\newblock An {$\infty$}-categorical approach to {$R$}-line bundles,
  {$R$}-module {T}hom spectra, and twisted {$R$}-homology.
\newblock {\em J. Topol.}, 7(3):869--893, 2014.

\bibitem{paper:AtiyahSegal}
Michael Atiyah and Graeme Segal.
\newblock Twisted {K}-theory.
\newblock {\em Ukr. Mat. Visn.}, 1(3):287--330, 2004.

\bibitem{paper:BaumHigsonSchick}
Paul Baum, Nigel Higson, and Thomas Schick.
\newblock On the equivalence of geometric and analytic {$K$}-homology.
\newblock {\em Pure Appl. Math. Q.}, 3(1, part 3):1--24, 2007.

\bibitem{paper:BOSW}
Paul Baum, Herv{\'e} Oyono-Oyono, Thomas Schick, and Michael Walter.
\newblock Equivariant geometric {$K$}-homology for compact {L}ie group actions.
\newblock {\em Abh. Math. Semin. Univ. Hambg.}, 80(2):149--173, 2010.

\bibitem{book:BlackadarOpAlg}
B.~Blackadar.
\newblock {\em Operator algebras}, volume 122 of {\em Encyclopaedia of
  Mathematical Sciences}.
\newblock Springer-Verlag, Berlin, 2006.
\newblock Theory of $C{^{*}}$-algebras and von Neumann algebras, Operator
  Algebras and Non-commutative Geometry, III.

\bibitem{book:Blackadar}
Bruce Blackadar.
\newblock {\em {$K$}-theory for operator algebras}, volume~5 of {\em
  Mathematical Sciences Research Institute Publications}.
\newblock Cambridge University Press, Cambridge, second edition, 1998.

\bibitem{paper:BCMMS}
Peter Bouwknegt, Alan~L. Carey, Varghese Mathai, Michael~K. Murray, and Danny
  Stevenson.
\newblock Twisted {$K$}-theory and {$K$}-theory of bundle gerbes.
\newblock {\em Comm. Math. Phys.}, 228(1):17--45, 2002.

\bibitem{paper:BrownStableIso}
Lawrence~G. Brown.
\newblock Stable isomorphism of hereditary subalgebras of {$C^*$}-algebras.
\newblock {\em Pacific J. Math.}, 71(2):335--348, 1977.

\bibitem{paper:BrownGreenRieffel}
Lawrence~G. Brown, Philip Green, and Marc~A. Rieffel.
\newblock Stable isomorphism and strong {M}orita equivalence of
  {$C\sp*$}-algebras.
\newblock {\em Pacific J. Math.}, 71(2):349--363, 1977.

\bibitem{paper:Cordier}
Jean-Marc Cordier.
\newblock Sur la notion de diagramme homotopiquement coh\'erent.
\newblock {\em Cahiers Topologie G\'eom. Diff\'erentielle}, 23(1):93--112,
  1982.
\newblock Third Colloquium on Categories, Part VI (Amiens, 1980).

\bibitem{paper:CuntzHigson}
Joachim Cuntz and Nigel Higson.
\newblock Kuiper's theorem for {H}ilbert modules.
\newblock In {\em Operator algebras and mathematical physics ({I}owa {C}ity,
  {I}owa, 1985)}, volume~62 of {\em Contemp. Math.}, pages 429--435. Amer.
  Math. Soc., Providence, RI, 1987.

\bibitem{paper:DadarlatP1}
Marius Dadarlat and Ulrich Pennig.
\newblock {A Dixmier-Douady theory for strongly self-absorbing $C^*$-algebras}.
\newblock \href{http://arxiv.org/abs/1302.4468}{arXiv:1302.4468} [math.OA] --
  to appear in J.\ Reine Angew.\ Math., 2013.

\bibitem{paper:DadarlatP2}
Marius Dadarlat and Ulrich Pennig.
\newblock {Unit spectra of $K$-theory from strongly self-absorbing
  $C^*$-algebras}.
\newblock \href{http://arxiv.org/abs/1306.2583}{arXiv:1306.2583} [math.AT] --
  to appear in Algebr.\ Geom.\ Topol., 2013.

\bibitem{paper:DadarlatWinterTriv}
Marius Dadarlat and Wilhelm Winter.
\newblock Trivialization of {$C(X)$}-algebras with strongly self-absorbing
  fibres.
\newblock {\em Bull. Soc. Math. France}, 136(4):575--606, 2008.

\bibitem{paper:DadarlatKK1}
Marius Dadarlat and Wilhelm Winter.
\newblock On the {$KK$}-theory of strongly self-absorbing {$C^*$}-algebras.
\newblock {\em Math. Scand.}, 104(1):95--107, 2009.

\bibitem{paper:FunctorialKSpectrum}
Ivo Dell'Ambrogio, Heath Emerson, Kandelaki Tamaz, and Ralf Meyer.
\newblock A functorial equivariant {$K$}-theory spectrum and an equivariant
  lefschetz formula.
\newblock \href{http://arxiv.org/abs/1104.3441v1}{arXiv:1104.3441v1} [math.KT],
  2011.

\bibitem{paper:DonovanKaroubi}
P.~Donovan and M.~Karoubi.
\newblock Graded {B}rauer groups and {$K$}-theory with local coefficients.
\newblock {\em Inst. Hautes \'Etudes Sci. Publ. Math.}, (38):5--25, 1970.

\bibitem{paper:EEK}
Siegfried Echterhoff, Heath Emerson, and Hyun~Jeong Kim.
\newblock {$KK$}-theoretic duality for proper twisted actions.
\newblock {\em Math. Ann.}, 340(4):839--873, 2008.

\bibitem{paper:FreedHopkinsTelemanI}
Daniel~S. Freed, Michael~J. Hopkins, and Constantin Teleman.
\newblock Loop groups and twisted {$K$}-theory {I}.
\newblock {\em J. Topol.}, 4(4):737--798, 2011.

\bibitem{paper:Joachim}
Michael Joachim.
\newblock {$K$}-homology of {$C^\ast$}-categories and symmetric spectra
  representing {$K$}-homology.
\newblock {\em Math. Ann.}, 327(4):641--670, 2003.

\bibitem{paper:JoachimEquiv}
Michael Joachim.
\newblock Higher coherences for equivariant {$K$}-theory.
\newblock In {\em Structured ring spectra}, volume 315 of {\em London Math.
  Soc. Lecture Note Ser.}, pages 87--114. Cambridge Univ. Press, Cambridge,
  2004.

\bibitem{paper:Kasparov}
G.~G. Kasparov.
\newblock Equivariant {$KK$}-theory and the {N}ovikov conjecture.
\newblock {\em Invent. Math.}, 91(1):147--201, 1988.

\bibitem{book:Lance}
E.~C. Lance.
\newblock {\em Hilbert {$C^*$}-modules}, volume 210 of {\em London Mathematical
  Society Lecture Note Series}.
\newblock Cambridge University Press, Cambridge, 1995.
\newblock A toolkit for operator algebraists.

\bibitem{book:Lurie}
Jacob Lurie.
\newblock {\em Higher topos theory}, volume 170 of {\em Annals of Mathematics
  Studies}.
\newblock Princeton University Press, Princeton, NJ, 2009.

\bibitem{book:Lurie2}
Jacob Lurie.
\newblock Higher algebra.
\newblock \url{http://www.math.harvard.edu/~lurie/papers/higheralgebra.pdf},
  2014.

\bibitem{book:MaySigurdsson}
J.~P. May and J.~Sigurdsson.
\newblock {\em Parametrized homotopy theory}, volume 132 of {\em Mathematical
  Surveys and Monographs}.
\newblock American Mathematical Society, Providence, RI, 2006.

\bibitem{paper:RieffelMorita}
Marc~A. Rieffel.
\newblock Morita equivalence for operator algebras.
\newblock In {\em Operator algebras and applications, {P}art {I} ({K}ingston,
  {O}nt., 1980)}, volume~38 of {\em Proc. Sympos. Pure Math.}, pages 285--298.
  Amer. Math. Soc., Providence, R.I., 1982.

\bibitem{book:RoerdamStoermer}
M.~R{\o}rdam and E.~St{\o}rmer.
\newblock {\em Classification of nuclear {$C^*$}-algebras. {E}ntropy in
  operator algebras}, volume 126 of {\em Encyclopaedia of Mathematical
  Sciences}.
\newblock Springer-Verlag, Berlin, 2002.
\newblock Operator Algebras and Non-commutative Geometry, 7.

\bibitem{paper:RosenbergCT}
Jonathan Rosenberg.
\newblock Continuous-trace algebras from the bundle theoretic point of view.
\newblock {\em J. Austral. Math. Soc. Ser. A}, 47(3):368--381, 1989.

\bibitem{paper:TomsAH}
Andrew~S. Toms.
\newblock Characterizing classifiable {AH} algebras.
\newblock {\em C. R. Math. Acad. Sci. Soc. R. Can.}, 33(4):123--126, 2011.

\bibitem{paper:TomsWinter}
Andrew~S. Toms and Wilhelm Winter.
\newblock Strongly self-absorbing {$C^*$}-algebras.
\newblock {\em Trans. Amer. Math. Soc.}, 359(8):3999--4029, 2007.

\bibitem{paper:Trout}
Jody Trout.
\newblock {On graded K-theory, elliptic operators and the functional calculus}.
\newblock {\em Illinois Journal of Mathematics}, 44(2):294--309, 2000.

\bibitem{paper:Wang1}
Bai-Ling Wang.
\newblock Geometric cycles, index theory and twisted {$K$}-homology.
\newblock {\em J. Noncommut. Geom.}, 2(4):497--552, 2008.

\bibitem{paper:Wang2}
Bai-Ling Wang.
\newblock Notes on twisted {$K$}-homology.
\newblock in preparation, 2015.

\bibitem{paper:WinterZStable}
Wilhelm Winter.
\newblock Strongly self-absorbing {$C^*$}-algebras are {$Z$}-stable.
\newblock {\em J. Noncommut. Geom.}, 5(2):253--264, 2011.

\end{thebibliography}

\end{document}